\documentclass[reqno,a4paper,11pt]{amsart}
\usepackage{amsmath,amsthm,amssymb,mathrsfs}
\usepackage[
left = 2.25cm,
right = 2.25cm,
top = 2.2 cm,
bottom = 2.2 cm
]{geometry}
\usepackage{graphicx,xcolor,tikz,mathtools}
\usepackage{bm,stmaryrd}
\usepackage{amsaddr}
\usepackage{enumitem}
\usepackage[toc,page]{appendix}
\usepackage[
pdfencoding=auto,
colorlinks = true,
citecolor = black,
linkcolor = black,
anchorcolor = black,
urlcolor = black, 
filecolor=black,
pdfkeywords={}
]{hyperref}
\usepackage{esint} 

\usepackage{todonotes}
\makeatletter
\define@key{todonotes}{Helmut}[]{%
	\setkeys{todonotes}{author=\textbf{Harald},inline,color=red!10}}%
\define@key{todonotes}{Andrea}[]{%
	\setkeys{todonotes}{author=\textbf{Andrea},inline,color=green!20}}%
\define@key{todonotes}{Tim}[]{%
	\setkeys{todonotes}{author=\textbf{Tim},inline,color=cyan!20}}%
\makeatother

\theoremstyle{plain}
\newtheorem{theorem}{Theorem}[section]

\newtheorem{lemma}[theorem]{Lemma}

\newtheorem{remark}[theorem]{Remark}
\theoremstyle{definition}

\numberwithin{equation}{section}

\newcommand{\bn}{\mathbf{n}}

\def\tE{\widetilde{\mathcal E}}

\newcommand{\bH}{\mathbf{H}}

\newcommand{\bP}{\mathbf{P}}

\newcommand{\bW}{\mathbf{W}}

\def\bL{\L}

\newcommand{\cE}{\mathcal{E}}

\def\mE{\mathcal E}

\def\argmin#1{\underset{#1}{\mathrm{argmin}}}

\renewcommand{\d}{\mathrm{d}}
\newcommand{\dx}{\,\d x}
\newcommand{\dt}{\,\d t}

\newcommand{\ddt}{\frac{\d}{\d t}}

\newcommand{\ptial}[1]{ \partial_{#1} }
\newcommand{\pt}{\ptial{t}}
\newcommand{\onehalf}{\frac{1}{2}}

\def\mA{\mathcal A}
\def\mB{\mathcal B}

\newcommand{\norm}[1]{\left\Vert #1 \right \Vert}

\newcommand{\normmm}[1]{\left\vert #1 \right\vert}

\def\Lps{ L^{p}(0,T;\W^{1-\frac1p,p}(\partial\Omega))}
\def\Wps{W^{\frac12-\frac1{2p},p}(0,T;\L^p(\partial\Omega))}
\def\Div{\mathrm{div}\,}
\DeclareMathOperator*{\esssup}{\mathrm{esssup}}

\def\w{\wedge}

\def\ovu{\overline u}

\def\d{{\rm d}}

\def\W{\boldsymbol{W}}

\def\L{\boldsymbol{L}}

\def\ddt{\frac{\d}{\d t}}

\def\mmu{\boldsymbol{\mu}}

\newcommand{\numberset}{\mathbb}
\newcommand{\N}{\numberset{N}}
\newcommand{\R}{\numberset{R}}

\def\bP{\boldsymbol{P}} 
\def\v{\boldsymbol{v}}

\def\u{{u}}

\def\phit_{\phi_t}

\def\tT{{T}}

\def\XT{X_{{T}}}
\def\YT{Y_{{T}}}
\def\norm#1{\left\Vert#1\right\Vert }

\def\ddt{\frac{d}{dt}}

\def\LLQ(#1,#2){L^{#1}(0,\widetilde{T};\L^{#2}(\Omega_0))}

\def\div#1{\text{div}_{#1}}

\def\non{\nonumber}

\def\multibold #1{\def\arg{#1}%
	\ifx\arg\pto \let\next\relax
	\else
	\def\next{\expandafter
		\def\csname #1#1#1\endcsname{{\boldsymbol #1}}%
		\multibold}%
	\fi \next}

\def\pto{.}

\def\multical #1{\def\arg{#1}%
	\ifx\arg\pto \let\next\relax
	\else
	\def\next{\expandafter
		\def\csname #1#1\endcsname{{\cal #1}}%
		\multical}%
	\fi \next}

\def\multimathop #1 {\def\arg{#1}%
	\ifx\arg\pto \let\next\relax
	\else
	\def\next{\expandafter
		\def\csname #1\endcsname{\mathop{\rm #1}\nolimits}%
		\multimathop}%
	\fi \next}

\multibold
qwertyuiopasdfghjklzxcvbnmQWERTYUIOPASDFGHJKLZXCVBNM.

\multical
QWERTYUIOPASDFGHJKLZXCVBNM.

\DeclarePairedDelimiter\ceil{\lceil}{\rceil}
\DeclarePairedDelimiter\floor{\lfloor}{\rfloor}

\allowdisplaybreaks[4]

\begin{document}
	
	\title[
    Weak and strong solutions for a class of quasilinear Allen--Cahn systems]
   {Weak and strong solutions\\ for a class of \\quasilinear Allen--Cahn systems}
	
	\author[H. Garcke, T. Laux, and A. Poiatti]{
		\small
		Harald Garcke$^\ast$, 
		Tim Laux$^{\dagger}$, and 
		Andrea Poiatti$^\ddagger$
	}
	
	\address{
		$^\ast$Fakult\"at f\"ur Mathematik,
		Universit\"at Regensburg,
		93053 Regensburg, Germany
	}
	\email{harald.garcke@ur.de}
	
	\address{
		$^{\dagger\dagger}$Institut f\"{u}r Mathematik \& Interdisziplin\"{a}res Zentrum f\"{u}r Wissenschaftliches Rechnen, Universit\"{a}t Heidelberg, Im Neuenheimer Feld 205, 69120 Heidelberg, Germany
	}
	\email{tim.laux@math.uni-heidelberg.de}
	
	\address{
		  $^\ddagger$Dipartimento di Scienze Matematiche, Fisiche e Informatiche, \\ Università degli Studi di Parma, 43124 Parma, Italy
	}
	\email{andrea.poiatti@unipr.it}

	

	\begin{abstract}
   We consider a quasilinear Allen--Cahn system which arises when the gradient energy term in the Ginzburg--Landau energy also contains zero order terms. 
   Such systems offer significant advantages in applications, since surface tensions and mobilities can be easily calibrated.
   The analysis for these systems is highly challenging, partly due to the fact that the gradient term in the energy is non-convex and since gradient terms appear quadratically in the weak formulation. This explains why an existence theory has been lacking for nearly thirty years.
   
   In this paper, we give the first existence and uniqueness results for such systems.   {Firstly,} we prove existence and uniqueness  of local-in-time strong solutions using the theory of maximal regularity. Here, non-standard techniques have to be applied due to the fact that linear constraints on the solution are involved and due to nonlinear boundary conditions.   {Secondly,} using a minimizing movement approach
   we show the existence of global-in-time weak solutions.  {Here, the} main difficulty arises from the fact that the underlying energy is not $\lambda$-convex. We overcome this issue by proving higher integrability of the gradient of the solution, first showing that solutions are bounded and then using an approach by Giaquinta and Modica. 
   This finally allows  us to pass to the limit in the time-discrete approximation.
    Using the de Giorgi interpolation technique, we are also able to show a  {sharp} energy decay property despite  the lack of convexity of the energy. 	
	\end{abstract}
	
	\maketitle

	\section{Introduction}
    Phase-field models provide a powerful framework to describe the evolution of interfaces. A crucial advantage of the phase-field methodology is that the interface does not need to be represented explicitly. In fact, a phase-field variable (also called order-parameter) is used to identify the phases as regions where the variable attains certain  values. The location of the interface is then given as the  region where the value of the phase-field variable changes its value quickly between these values.   {Hence,} the approach is of Eulerian type, meaning that, in particular, it is not necessary to follow the interface with the help of a parameterization.
    An attractive feature of phase-field models is that solutions are typically smooth  {even in} the case when the topology of the interfaces changes. 

    Thus, phase-field approaches  have been used in a multitude of situations modeling such diverse phenomena as grain growth, phase separation, Ostwald ripening, solidification, two-phase flow, phase transitions in metals, and interfaces in image analysis. In the simplest case, phase field models are based on a Ginzburg--Landau energy (also called Cahn--Hilliard energy)
    $$
    \int_\Omega\frac\varepsilon 2\normmm{\nabla u}^2\dx+\frac 1\varepsilon\int_\Omega \Psi(u) \dx,
    $$
    where $\Omega\subset \mathbb R^d$ is a  bounded  domain with suitable smoothness,  $\varepsilon>0$ is a small parameter proportional to the thickness of the diffuse interface layer, and $\Psi$ is a double-well potential having the property $\Psi(\pm 1)=0$ and $\Psi(u)>0$
    if  $u\notin \{ -1,1\} $. The $\Psi$-term in the energy therefore penalizes deviations 
    from the pure phases which attain the values $\pm 1$. 
The term $\frac\varepsilon 2\normmm{\nabla u}^2$
penalizes  rapid spatial changes in $u$ and will hence enforce that the size of the interfaces cannot become too large. One major result on the Ginzburg--Landau energy is that $\int_\Omega\frac\varepsilon 2\normmm{\nabla u}^2\dx+\frac 1\varepsilon \int_\Omega \Psi(u) \dx$ as $\varepsilon \to 0$ in the sense of $\Gamma$-limits converges to a multiple of the perimeter functional; see \cite{Modica87}.
The number multiplying the perimeter functional in the $\Gamma$-limit is the surface energy density and can  easily be explicitly computed from the double well function $\Psi$. In the proof of the $\Gamma$-limit one realizes that the thicknesses of typical interfaces are proportional to $\varepsilon$ and sending $\varepsilon$ to zero hence leads to a sharp interface limit of description using diffuse interfaces.

    In many applications several phases occur, and hence multiphase-field models  need to be studied. In this case, a vector of phase-fields $u$ is used. 
    In the vector $u=(u_1,u_2,\ldots,u_N)^T$
    of phase-fields, the component $u_i$ accounts for the local fraction of the  {phase} $i$ and due to this, the constraint
    $$\sum_{i=1}^N u_i=1$$
    needs to be imposed. The simplest version of a multiphase Ginzburg--Landau energy is given as
    $$\frac\varepsilon 2\sum_{i=1}^N\int_\Omega\normmm{\nabla u_i}^2\dx+\frac 1\varepsilon\int_\Omega \Psi(u) \dx.$$
    Here, $\Psi$ is a multiwell potential
   that  has a  minimum value $0$ at the minimum points $u=\mathbf e_i$, $i=1,\ldots,N$, where $\mathbf e_i$ are the standard basis vectors in $\mathbb R^N$.
     {Also for this multiphase Ginzburg--Landau energy the  $\Gamma$-limit can be identified, see \cite{Baldo90}. However, in general, the surface tensions (which depend on the two neighboring phases of an interface) cannot be computed explicitly.
    In applications, the surface tensions of the sharp interface formulation are typically known,} and it is desirable to calibrate the phase-field model in such a way that given surface tensions are recovered. In addition, in evolution problems of multi-phase systems, each pair
    of interfaces also has a specific   {mobility} coefficient.
In order to have enough flexibility   {to be able to calibrate the phase-field model to given matrices of surface tensions $(\sigma_{ij})_{i,j=1\ldots N}$ and mobilities $(\mu_{ij})_{i,j=1,\ldots,N}$}, Steinbach et al.\
\cite{STEINBACH1996} introduced a Ginzburg--Landau energy that is based on the anti-symmetric terms $$u_i\nabla u_j -u_j\nabla u_i$$
which are approximations of the directions normal to the interface and hence allows it to identify each of the interfaces individually. In fact, in \cite{STEINBACH1996}
the energy is expressed as 
\begin{equation}\label{eq:fullenergy}
\mathcal E(u):=\int_\Omega f(u,\nabla u)\dx +\int_\Omega \Psi(u)\dx, 
\end{equation}
with 
\begin{equation}\label{eq:fullenergy2}
f(u,\nabla u):=\sum_{i,j=1}^N\onehalf\varepsilon_{ij}^2\normmm{u_i\nabla u_j-u_j\nabla u_i}^2,
\end{equation}
which penalizes the energy individually at each of the interfaces connecting $\bf e_i$ and $\bf e_j$. 
The potential $\Psi$ as above fulfills the property
that the minimum value $0$ is attained exactly at the minimal points $u=\mathbf e_i$, $i=1,\ldots,N$. 
We define 
    \begin{align*}
        \Sigma:=\{x\in \mathbb R^N: \sum_{i=1}^Nx_i=1\},
    \end{align*}
    as well as the tangent space
       \begin{align*}
       T \Sigma:=\{x\in \mathbb R^N: \sum_{i=1}^Nx_i=0\},
    \end{align*}
    and the orthogonal projection  
     $\bP$  onto $T\Sigma$, i.e., setting $\mathbf e:=(1,1,\ldots,1)^T$, we have $$\bP:=Id-\frac1N \mathbf e\otimes \mathbf e.$$
Also, the $\Gamma$-limit of the  energy \eqref{eq:fullenergy}   {with} the constraint that values of $u$ lie on $\Sigma$  can be identified, see \cite{Bellettini}. In \cite{GarckeStoth} 
the  $L^2$ gradient flow of this energy was derived, taking the sum constraint $\sum_{i=1}^N u_i=1$ into account, and the authors obtained  the following Allen--Cahn system for $u\in \Sigma$: 
\begin{align}
\label{eq1}\partial_t u+\bP \big(-\div{}{(f_{,\nabla u}(u,\nabla u))}+ f_{,u}(u,\nabla u)+\Psi_{,u}(u)\big)&=0,\quad\text{ in }\Omega\times(0,\infty),\\
\bP (f_{,\nabla u}(u,\nabla u))\cdot \bn&=0, \quad\text{ on }\partial\Omega\times(0,\infty),\\
u(0)&=u_0,\quad\text{ in }\Omega,\label{eq3}
\end{align}
where $f_{,u}$ and $f_{,\nabla u}$ denote the partial derivatives of $f$ with respect to $u$ and  $\nabla u$, respectively,    {and} $\bn$ is the outward unit normal to $\partial\Omega$. The above initial boundary value problem has been used 
subsequently to model several multi-phase evolution problems.   {In} particular, the fact that surface tensions for the individual interfaces can be   {tuned} appropriately makes the model very powerful, see 
\cite{GarckeNestlerStothSIAM,GarckeNSIFB, NESTLER2000114, GarckeNStinnerSIAM} for practical applications of the approach.

Ever since its introduction in 1998 it has not been possible to 
 establish a theory of existence and uniqueness for the Allen--Cahn system \eqref{eq1}--\eqref{eq3}. Establishing such a theory is difficult due to the following 
 facts:
 \begin{itemize}
 \item The term $f_{,u}(u,\nabla u)$ is quadratic in $\nabla u$, making  it very challenging to apply approaches  based on energy methods.
 \item Although the gradient energy density
 $f(u,\nabla u)$ is convex in $u$ and $\nabla u$ individually,   {it} is \textit{not} jointly convex. It even turns out that the 
 energy is not $\lambda$-convex, i.e., one cannot add a quadratic term to make it convex, which makes it difficult to apply gradient flow approaches such as the de Giorgi minimizing movement strategy (\cite{AGS,degiorgi}). 
 \item The constraint $u\in \Sigma$ and the highly nonlinear structure of the boundary conditions make it difficult to apply semigroup approaches to show local-in-time well-posedness.
\end{itemize}
  {It is the goal of this paper to prove, for the first time for a quasilinear parabolic system of the form~\eqref{eq1}--\eqref{eq3},  
\begin{itemize}
\item local-in-time existence and uniqueness of strong solutions; 
\item global-in-time existence of weak solutions.
\end{itemize}
}
Before we discuss the main findings of this paper, let us discuss other properties of the initial-boundary value problem \eqref{eq1}--\eqref{eq3}. 
Summing the $N$ equations in \eqref{eq1}
gives $$
\partial_t \sum_{i=1}^Nu_i\equiv 0,
$$
and thus, if $u_0\in \Sigma$, then also $u(t)\in  \Sigma$ for any $t\geq 0$. We now focus  on the following energy, where we set $f(u,\nabla u):=\sum_{i,j=1}^N\onehalf
\normmm{u_i\nabla u_j-u_j\nabla u_i}^2$: 
    \begin{align}
        E(u):=\underbrace{\sum_{i,j=1}^N\onehalf\int_\Omega
        \normmm{u_i\nabla u_j-u_j\nabla u_i}^2\dx}_{\sum_{k=1}^d \int_\Omega E_1(u,\partial_{x_k} u)\dx}+\int_\Omega \Psi(u) \dx,\label{ene}
    \end{align}
 where we have defined
    \begin{align*}
        E_1(u,p):=\frac12\sum_{i,j=1}^N
        \normmm{u_i p_j-u_jp_i}^2,\quad \forall u\in \Sigma, p\in T\Sigma,
    \end{align*}
    Notice that 
    we choose for simplicity $\varepsilon_{ij}=\varepsilon= 1$, so that we have 
    \begin{align}
    E_1(u,p)=\frac 12\normmm{u\wedge p}^2,\quad f(s,X)=\frac 1 2 \sum_{k=1}^d E_1(s,X_{k\cdot}),\label{formused}
    \end{align}
    where we set 
    $$
    u\wedge p=\sum_{i,j=1}^N(u_ip_j-u_jp_i)e_i\wedge e_j,
    $$
    and $e_i\wedge e_j$ is the exterior product between the elements $i$, $j$ of the canonical basis of $\R^N$. 
We will discuss later which parts of the paper can be generalized to 
a matrix $(\varepsilon_{ij})$ as in \eqref{eq:fullenergy2}.
All parts generalize to an arbitrary $\varepsilon> 0$.

After having introduced the notation and main assumptions, we formulate  the main
results in Section \ref{mainresults}. They consist of   {Theorem~\ref{thm:existenceweak}} stating the existence of a global weak solution and   {Theorem~\ref{thm1}} on the local existence and uniqueness of a strong solution. In the same section, we also show that the energy density \eqref{eq:fullenergy2}
 is not jointly convex. In Section \ref{exstrong} we prove the local existence and uniqueness result using the theory of maximal regularity as presented, e.g., in \cite{PS2016}. The proof uses a contraction mapping principle and is made more difficult by the nonlinear boundary conditions.  The existence of weak solution is shown in Section \ref{exweak} using a minimizing movement scheme. Due to quadratic terms in the gradient   {in the Euler-Lagrange equation}, a higher-order integrability of the gradient needs to be shown.    {As the elliptic operator depends
 on the values of $u$ itself, we need to ensure that solutions remain sufficiently bounded, which  leads to restrictions on the  position of the minima of the multi-well potential. This, however, only leads to restrictions when the number of components \textcolor{black}{is} larger than four.}

\section{Notations and main assumptions}
Let $\Omega\subset \R^d$, $d\geq 1$, be a $d$-dimensional bounded domain with smooth boundary.
 We denote the usual Sobolev spaces by $W^{k,p}(\Omega )$%
	, where $k\in \mathbb{N}$ and $1\leq p\leq +\infty $, with norm $\Vert \cdot
	\Vert _{W^{k,p}(\Omega )}$. The Hilbert space $W^{k,2}(\Omega )$ is identified with $H^{k}(\Omega )$ with norm $\Vert \cdot \Vert _{H^{k}(\Omega )}$.
	Furthermore, given a (real) vector space $X$, we set $\mathbf{X}=X^r$, where $r\in \N$ is the number of components of the corresponding vector fields. We then denote by $(\cdot
	,\cdot )$ the inner product in $L^{2}(\Omega )$ and by $\Vert \cdot \Vert $
	the corresponding induced norm. By $(\cdot ,\cdot )_{X}$ and $\Vert \cdot
	\Vert _{X}$ we represent the canonical inner product and its induced norm in a generic real Hilbert
	space $X$, respectively.  

	We then introduce the Gibbs simplex
	\begin{equation}
    \mathbf{G}:=\left\{   {u}\in \mathbb{R}^{N}:%
	\sum_{i=1}^{N}u_{i}=1,\quad u_{i}\geq 0,\quad i=1,\ldots
	,N\right\} ,  \label{Gibbs}
	\end{equation}
    which is the set in which all  meaningful physical values attained by $u$ should lie in.

We assume that the potential $\Psi\in C^2(\R^N;\R^+)$, with $\Psi\geq -C$ for some $C>0$, can be decomposed as 
\begin{align}
\Psi=\Psi_1+\Psi_2,  
\end{align}
where $\Psi_1$ is convex and $\Psi_2(s):=\frac12 s\cdot \mathbf As$, with $\mathbf A\in \R^{N\times N}$ a symmetric $N\times N$ matrix.  
An example of such a potential is a smooth version of the well-known multi-well potential introduced in \cite{EL} (see also \cite{GGPS,GarckeStoth,CAC,Boyer2014, MR2553473}), which reads: for any $u\in \Sigma$,
\begin{align*}
\Psi(u)=\Psi_1(u)+\Psi_2(u):=\sum_{i=1}^N \psi_i(u_i)-\frac12u\cdot \mathbf Bu,
\end{align*}
with $\psi_i(s)$ as a regular potential such that $\psi_i(s)=0$ if and only if $s=0$, for any $i=1,\ldots,N$,  and $\mathbf B\in \R^{N\times N}$ a symmetric and positive definite matrix so that $\Psi\geq -C$.
{  
For further examples, we refer to Section 3 of \cite{GarckeNestlerStothSIAM}, and \cite{GarckeNStinnerSIAM, NestlerGarckeStinner}.}

When proving the existence of weak solutions, we will need the following, additional assumption: there exists $M<
\sqrt{\frac 4N}$ such that
\begin{align}
\Psi(T_M(\mathbf P u)+\tfrac1N\mathbf e)\leq \Psi(u),\quad \forall u\in \Sigma,\label{assbasic}
\end{align}
where the   {truncation $T_M$ is given by}
\begin{align}
T_M:T\Sigma\to T\Sigma,\qquad T_M(s):=\begin{cases}
s,\quad\text{ if }\normmm{s}\leq M,\\
\frac{M}{\normmm{s}}s,\quad\text{ if }\normmm{s}>M.
\end{cases}\label{TM}
\end{align}
This assumption has a simple geometric interpretation.   {  
Roughly speaking, the potential $\Psi$ needs to be large enough far away from the center of the Gibbs simplex}. If we consider $\mathbf z_0=\frac1N\mathbf e\in \Sigma$ together with $C_r:=\partial B_r(\mathbf z_0)\cap \Sigma$ (intersection between the $N$-dimensional ball of radius $r>0$ centered at $\mathbf z_0$, and $\Sigma$, corresponding to an $(N-1)$-dimensional ball), the potential $\Psi$ has to be increasing over $C_r$ for $r>M$, with $M<\sqrt{\frac 4N}$. If one assumes that $\Psi( s)\to \infty$ for $\normmm{s}\to \infty$, an example of such potential is a potential having wells in the Gibbs simplex not far from $\frac1N\mathbf e$, increasing along $C_r$ for $r>0$ sufficiently large. {  Notice that, when $N\leq4$, we have $\mathbf G\subset B_M(z_0)\cap \Sigma$ (see also Remark \ref{smallness}), so that the constraint is not actually active. Indeed, since the zeroes of the potential $\Psi$ have to be contained in the Gibbs simplex $\mathbf G$ to be physically relevant, it is enough to modify $\Psi$ outside $\mathbf G$ in order to comply with \eqref{assbasic}.}
  
We then define the energy $\cE_1$ as
\begin{align*}
    \cE_1(u,v):=\sum_{i=1}^d \int_\Omega E_1( u,\partial_{x_i} v)\dx, \quad \forall u\in \bL^\infty(\Omega; \Sigma),\quad \forall v\in \bH^1(\Omega;\Sigma),
\end{align*}
and introduce the quantity 
\begin{align*}
    \widetilde {\mathcal E}(u,v)=\cE_1(u,v)+{\int_\Omega \Psi(u)\dx},\quad \forall u\in \bL^\infty(\Omega; \Sigma)\quad \forall v\in \bH^1(\Omega;\Sigma).
\end{align*} 

We then denote by $\delta \mathcal E_1(\cdot,v)(u)$ the variation of $\mathcal E_1$ with respect to its first argument, evaluated at $u\in \bL^\infty(\Omega;\Sigma)$, for a fixed $v\in \bH^1(\Omega;\Sigma)$. 

On the other hand, for any fixed $u\in \bL^\infty(\Omega;\Sigma)$, we denote by $\delta \cE_1(u,\cdot)(v)$ the variation of $\cE_1$ with respect to its second argument, evaluated at $v\in \bH^1(\Omega;\Sigma)$. The same notations are used for $\widetilde\cE$.
 
Note that the total energy $ E$ can be seen as 
\begin{align}
 E(u)=\tE(u,u),
\end{align}
and this justifies the notation. We also set, for notational simplicity,
\begin{align}\label{E1def}
\widehat \mE_1(u):=\mathcal E_1(u,u). 
\end{align}

To make the notation more concrete, given $u \in \bL^\infty(\Omega;\Sigma)$, $v \in \bH^1(\Omega;\Sigma)$, we have
\begin{align*}
\langle \delta \mathcal E_1(\cdot,v)(u),h\rangle:=\sum_{k=1}^d\int_\Omega({u\wedge \partial_{x_k}v},h\wedge \partial_{x_k}v)\dx,\quad \forall h\in \bL^\infty(\Omega;T\Sigma), 
\end{align*}
as well as
\begin{align*}
 \langle \delta\mathcal E_1(u,\cdot)(v),h\rangle=\sum_{k=1}^d\int_\Omega({u\wedge \partial_{x_k}v},u\wedge \partial_{x_k}h)\dx,\quad \forall h\in \bH^1(\Omega;T\Sigma).
\end{align*}

\section{Main results}\label{mainresults}
In this section, we state our main results. First, we give an existence result of global weak solutions as follows
\begin{theorem}[Existence of weak solutions]
\label{thm:existenceweak}Let $\Omega$ be a bounded domain with Lipschitz boundary, $u_0\in \bH^1(\Omega;\GGG)$, with $\normmm{\bP u_0}\leq M$, $M^2<\frac{4}N$, almost everywhere in $\Omega$, and assume that the potential $\Psi$ satisfies \eqref{assbasic}. Then there exists a weak solution $u:\Omega\times[0,\infty)\to \Sigma$, such that 
\begin{align}
&\normmm{\bP u}\leq M,\text{ a.e. in } \Omega\times (0,\infty),\\& 
u\in L^\infty(0,\infty;\bH^1(\Omega;\Sigma))\cap H^1(0,\infty;\L^2(\Omega;\Sigma)),
\end{align}
and there exists $r>0$, which does not depend on $u$, such that for all $T>0$
\begin{align}
u\in L^r(0,T;\bW^{1,r}(\Omega;\Sigma)).\quad 
\label{regularity}
\end{align}
Moreover, $u$ satisfies, for any $w\in \bH^1(\Omega;T\Sigma)\cap\bL^\infty(\Omega)$,
\begin{align}
&\int_\Omega \partial_t{u}\cdot w\dx+\langle\delta \mathcal E_1(\cdot,u)(u)+\delta \mathcal E_1(u,\cdot)(u),w\rangle\non\\&+\int_\Omega\Psi'(u)\cdot w\dx=0 ,\quad \text{for a.a. }t>0,\label{inclusion2}
\end{align}
together with the sharp energy dissipation inequality
\begin{align}
\mathcal E(u(t))+\int_s^t\norm{\pt u(\tau)}_{\bL^2(\Omega)}^2\d\tau\leq \mathcal E(u(s)), 
\label{energyineq1}
\end{align}
$\text{ for any }t>0 ,\text{and almost any }0\leq s<t,\text{ with }s=0 \text{ included}$.
\end{theorem}
\begin{remark}
\label{smallness}
The regularity \eqref{regularity} is critical to prove the existence of solutions, as $r=2$ would not be enough to pass to the limit in the approximation scheme. Also, we note that, as the potential is smooth, we are not able to show that $u\in \GGG$, but we only obtain that $u\in\Sigma$ is bounded. The smallness condition related to $M^2<\frac4N$, crucial to obtain the higher-order integrability \eqref{regularity}, is an active constraint only for $N>4$. Indeed, up to $N=4$ we have $\min\{\tfrac4N,1-\tfrac 1N\}=1-\tfrac1N$ and  $\GGG\subset B_{\sqrt{1-\frac1N}}(\frac1N\mathbf e)$, so that $u_0\in \GGG$ already guarantees the constraint for $N\leq4$, since we can always choose $M^2\geq 1-\frac 1N$. The case $N=5$ is on the threshold, in the sense that it is the case when $\frac 4N=1-\frac 1N$, and $u_0\in \mathbf G$ is admissible if $u_0(x)$ is not a pure phase in any $x\in \Omega$. The constraint on $M$ and $u_0$ becomes stricter only when $N>5$.
\end{remark}

\begin{remark}[On the joint convexity of the energy $\mE_1$ for $N>2$]
\label{noconvex}
We crucially observe that the quantity $E_1(u,p)$, $u\in \Sigma$, $p\in T\Sigma$, is convex in each of the two arguments, although it is not jointly convex if $N>2$. Indeed, let us compute, for $v,w,q,r\in T\Sigma$,
\begin{align*}
    &\partial_{u}E_1(u,p)(v)=(u\wedge  p, v\wedge p),\quad 
    \partial_{p}E_1(u,p)(q)=(u\w p,u\w q),\\&
    \partial_{uu}E_1(u,p)(v,w)=(v\wedge p, w\wedge p),\quad
    \partial_{pp}E_1(u,p)(q,r)=({u \wedge q}, u\w r)\\&
     \partial_{pu}E_1(u,p)(v,r)=({u \wedge r}, v\w p)+(u\w p, v\w r),\quad
     \partial_{up}E_1(u,p)(q,w)=({w \wedge p}, u\w q)+(u\w p, w\w q),
\end{align*}
so that the quadratic form associated to the second variation reads
\begin{align}
S(u,p)(v,q):=\normmm{v\w  p}^2+\normmm{u\w q}^2+2(u\w q,v\w p)+2(u\w p,v\w q).
    \label{second}
\end{align}
Recalling the formula
\begin{align}
(u\w p,v\w q)=(u,v)(p,q)-(u,q)(p,v),
    \label{sums}
\end{align}
it is immediate to infer that $E_1$ is convex separately in each argument (just by choosing $v=0$ or $q=0$, respectively, and observing that the quadratic form is positive semidefinite), but not jointly convex if $N>2$. Namely, let us consider the following counterexample: fix $u\in \Sigma$,  and set $q=\bP u\in T\Sigma$. Choose then $v\perp q$ ($v\not=0$, which is possible if $N>2$) such that $v\in T\Sigma$, and define $p=\beta v\in T\Sigma$, $\beta>0$. Then it holds 
\begin{align*}
    S(u,p)(v,q)=\normmm{\bP u}^2\left({\frac 1N}-2\beta \normmm{v}^2\right),
\end{align*}
and this quantity can be made arbitrarily negative as $\beta\to \infty$. Note that even if $u,v$ are bounded, $p,q$ can be arbitrarily large and the same counterexample holds. This means that $\mE_1$ cannot be jointly convex. Due to this lack of joint convexity we need to use the higher integrability result \eqref{regularity} to pass to the limit in an approximation scheme. Additionally, we have to resort to the De Giorgi interpolation in the scheme to retrieve the sharp energy dissipation inequality \eqref{energyineq1}.
\end{remark}
\begin{remark}[Differences   {to} harmonic maps in 2D]
We point out that, in two dimensions, if we considered the geometric constraint $\normmm{u}=1$ as for harmonic maps, the quantity $(u_i\nabla u_j-u_j\nabla u_i)\nabla u_i$, $i,j=1,\ldots,N$, responsible for the main nonlinearity generating the difficulties in passing to the limit in the approximation scheme, can actually be solved. Indeed, it can be seen (cf. \cite{Helein}) that there exists $B_i^j$ such that $\nabla ^\perp B_i^j=u_i\nabla u_j-u_j\nabla u_i$, where $\nabla^\perp =(-\partial_{x_2},\partial_{x_1})$. This means that we can rewrite the nonlinearity as $(u_i\nabla u_j-u_j\nabla u_i)\nabla u_i=\nabla^\perp B_i^j\cdot \nabla u_i$, and by the well known div-curl compensated compactness lemma we can pass to the limit. In our case,  accounting for the linear constraint $u\in\Sigma$, the situation is not as fortunate, since the term $u_i\nabla u_j-u_j\nabla u_i$ is not divergence free and its Helmholtz decomposition contains a gradient term $\nabla \psi^j_i$. The product $\nabla \psi_i^j\cdot \nabla u_i$ cannot apparently pass to the limit in a straightforward way by a structural compensated compactness argument.
\end{remark}

In addition to the existence of weak solutions, we also propose a short-time existence and uniqueness result of strong solutions in the case of a fully general smooth potential $\Psi_1\in C^2(\R^N)$ and \textit{without} any restriction on the size of the initial datum. In particular, we have
\begin{theorem}
    Let $\Omega\subset \mathbb R^d$ be a bounded domain with $C^2$ boundary, and let $\Psi$ be a smooth (at least $C^2$) potential. If $\u_0\in \W^{2-\frac2p,p}(\Omega;\Sigma)$, with $p>d+2$, such that $(\normmm{u_0}^2Id-\bP u_0\otimes \bP u_0)\nabla u_0\cdot \mathbf n=0$ on $\partial\Omega$, then there exists a $T>0$ and a unique strong solution to \eqref{eq1}--\eqref{eq3}, with $f(u,\nabla u)=\frac12\sum_{k=1}^d\normmm{u\wedge \partial_k u}^2$, such that 
    $$
    u\in L^p(0,T;\W^{2,p}(\Omega;\Sigma))\cap W^{1,p}(0,T;\L^p(\Omega)).
    $$
    \label{thm1}
\end{theorem}
The proof of this result relies on the theory of maximal regularity for analytic operators, and it is based on linearization and the contraction mapping theorem. Since $\Omega$ is a bounded domain with sufficiently smooth boundary, the most technical part is to address the validity of the so-called Lopatinskii-Shapiro conditions \cite{PS2016}.

\section{Proof of Theorem \ref{thm1}. Existence of local strong solutions}\label{exstrong}
We first state the following theorem about the $L^p$ maximal regularity of a specific operator which will appear in the proof of the existence of strong solutions.
\begin{theorem}
\label{Maxreg}
    Let $\Omega$ be a bounded domain of class $C^2$, assume $\ovu\in \mathbf C^{0,l}(\overline\Omega;\Sigma)\cap W^{1-\frac1p,p}(\partial\Omega)$, $l\in(0,1]$, and $\widetilde u_0\in W^{2-\frac2p,p}(\Omega;T\Sigma)$, for some $p\in(1,\infty)$. Consider the operator $\mathcal Au:=- (A_{ij}\partial_{kk} u_j)_{i=1}^N$, where $A(x):=\normmm{\ovu}^2Id-\bP\ovu\otimes \bP\ovu$. Define also $\mathcal Bu:=(A(x)\nabla u)\cdot \mathbf n$ on $\partial\Omega$. Then there is $\omega_0 \in \mathbb R$  such that for each $\omega>\omega_0$ the system
    \begin{align}
    \begin{cases}
   \partial_t u+\mathcal Au+\omega u=f,\quad&\text{in }\Omega,\\
    \mathcal Bu=g,\quad&\text{on }\partial\Omega,\\ u(0)=\widetilde u_0,\quad&\text{in }\Omega,\end{cases}
        \label{system}
    \end{align}
    with $g\in W^{\frac12-\frac1{2p},p}(\mathbb R^+;\L^p(\partial\Omega;T\Sigma))\cap L^{p}(\mathbb R^+;\W^{1-\frac1p,p}(\partial\Omega))$, $g(0)=\mathcal B u(0)$, and $f\in L^p(\mathbb R^+,\L^p(\Omega;T\Sigma))$,
    admits a
unique solution $u$ in the class $W^{1,p}(\mathbb R^+;L^p(\Omega;T\Sigma))\cap L^p(\mathbb R^+;W^{2,p}(\Omega))$. 
\end{theorem}
\begin{proof}[Proof of Theorem \ref{Maxreg}]
We want to apply \cite[Theorem 6.3.2]{PS2016}. To this aim we set $E=T\Sigma$, which is a finite dimensional Hilbert space, and thus of class $\mathcal H\mathcal T$ (see \cite{PS2016}). Then we write the operator $\mathcal A\in \mathcal L(E)$ as 
\begin{align*}
  \mathcal A(x,D)u:=(\normmm{\ovu}^2Id-\bP\ovu(x)\otimes  \bP\ovu(x))D_{kk}u,
\end{align*}
where $D=-i(\partial_1,\partial_2,\ldots,\partial_d)$, and $D_{kk}=D_kD_k$. From now on we will always use, whenever possible, Einstein summation convention.
Also, we can write $\mathcal B\in \mathcal L(E)$ as
\begin{align*}
    \mathcal B(x,D)=iA(x)D_ku\ \mathbf n_k\in E.
\end{align*}
We now aim at showing that the couple $(\mathcal A(x,D),\mB(x,D))$ is uniformly normally elliptic according to \cite[Definition 6.3.1]{PS2016}. First, $\mA$ is normally elliptic for any $x\in \overline\Omega$, since it is strongly elliptic, as 
\begin{align}
    (\mathcal A(x,\xi)v,v)\geq \frac1N\normmm{\xi}^2\normmm{v}^2,\label{ctrl1a}
\end{align}
for every $\xi\in \mathbb R^N$, $\xi\not=0$, for every $v\in E$ and for every $x\in \overline\Omega$. Indeed, recalling that $\ovu\in \Sigma$, by Cauchy-Schwarz inequality,
\begin{align*}
     (\mathcal A(x,\xi)v,v)&=((\normmm{\ovu}^2Id-\bP\ovu\otimes  \bP\ovu)\normmm{\xi}^2v)\cdot v\\&
     =\normmm{\xi}^2(\normmm{\ovu}^2\normmm{v}^2-(\bP\ovu\cdot v)^2)\\&
     =\normmm{\xi}^2(\normmm{v}^2\normmm{\bP\ovu}^2+\frac1N\normmm{v}^2-(\bP\ovu\cdot v)^2)\\&
     \geq \frac1N\normmm{\xi}^2\normmm{v}^2.
\end{align*}
Secondly, we need to prove that the Lopatinskii-Shapiro condition holds for any fixed $x\in \partial\Omega$. In our specific case we aim at verifying that, for any $x\in \partial\Omega$, for any $\lambda\in \mathbb C$ with $\mathrm{Re}\lambda\geq 0$ and any $\xi,\nu\in \mathbb R^d$, $(\xi,\lambda)\not=(0,0)$, $\normmm{\nu}=1$, $\xi\cdot \nu=0$,
\begin{align}
&\lambda u(y)+\mathcal A(x,\xi+D_y\nu)u(y)=0,\quad y>0,\nonumber\\&
\mathcal B(x,\xi+D_y\nu)u(0)=0
    \label{shapiro}
\end{align}
has exactly one solution $u\in C_0(\mathbb R^+;E),$ i.e., a continuous function vanishing at infinity. Recalling that $\nu\perp \xi$ and $\normmm{\nu}=1$, the condition can be reshaped as 
\begin{align*}
    &\lambda u(y)+A(x)(\normmm{\xi}^2+D_yD_y)u(y)=0,\quad y>0,\\&
    A(x)D_y u(0)=0.
\end{align*}
We can now multiply the equation by $u$ and integrate over $(0,\infty)$, to obtain, after an integration by parts, exploiting the condition at $y=0$ (recalling that $u$ is vanishing at infinity, which for a solution of an ODE means that $u$ and its derivatives converge to zero exponentially fast), and after taking the real parts, 
\begin{align*}
   \mathrm{Re} \lambda \int_0^\infty\normmm{u(y)}^2\d y+\mathrm{Re}\int_0^\infty A(x)(\xi_j+D_y\nu_j)u(y)\cdot (\xi_j+D_y\nu_j)u(y)\d y=0,
\end{align*}
which implies, since $\mathrm{Re}\lambda\geq0$, 
\begin{align}
    \mathrm{Re}\int_0^\infty A(x)(\xi_j+D_y\nu_j)u(y)\cdot(\xi_j+D_y  \nu_j)u(y)\d y\leq 0.\label{c1}
\end{align}
By a similar argument as to obtain \eqref{ctrl1a}, we infer
\begin{align*}
    A(x)v\cdot v\geq \frac 1N \normmm{v}^2,\quad \forall v\in \mathbb C^N,\quad \sum_{k=1}^Nv_k=0,
\end{align*}
so that inequality \eqref{c1} actually gives 
$$
\frac1N\sum_{j}\int_0^\infty\normmm{(\xi_j+D_y\nu_j)u}^2\d y=\frac1N\int_0^\infty(\normmm{\xi}^2\normmm{u}^2+\normmm{D_yu}^2)\d y\leq 0,$$
entailing that $u\equiv 0$ for any $y>0$, since it must be $u\in C_0(\mathbb R^+;E)$. This shows that the Lopatinskii-Shapiro condition is satisfied for any $x\in \partial\Omega$.

Since also $A\in C^{0,l}(\overline\Omega;\mathcal L(E))\cap W^{1-\frac1p}(\partial\Omega;\mathcal L(E))$ due to the assumption on $\ovu$, and since $\Omega$ has $C^2$ boundary, conditions (rA) and (rB) (with $p=q$, $r_{k}=p$, $m=1$) of \cite[Theorem 6.3.2]{PS2016} are also satisfied, and thus we can finally apply the theorem to obtain a solution $ u(x)\in E$ for any $x\in \overline\Omega$, with the  regularity required in Theorem \ref{Maxreg}, thus concluding its proof.

\end{proof}
We can now pass to proving our main   {local} existence theorem. 
\begin{proof}[Proof of Theorem \ref{thm1}]
Let us note that $\u_0\in \W^{2-\frac2p,p}(\Omega)\hookrightarrow \W^{1,2p}(\Omega)$ for $p\geq \tfrac{d+4}2$ and thus also for $p>d+2$. Additionally, the condition $p>d+2$ allows the embedding $\W^{2-\frac2p,p}(\Omega)\hookrightarrow \W^{s,p}(\Omega)$, $s>1+\tfrac dp$, and note that, by standard embedding results, $\W^{s,p}(\Omega)\hookrightarrow \mathbf C^{1,l}(\overline\Omega)$ for  $l=s-1-\frac d{p}\in(0,1)$. Given $T\in(0,T_0)$, {where $T_0>0$ is chosen arbitrarily large}, we introduce the spaces
	$$
	Z_T:=L^p(0,T;\W^{2,p}(\Omega;T\Sigma))\cap W^{1,p}(0,T;\L^p(\Omega)),
	$$
    $$
    \widetilde W_T:=W^{\frac12-\frac1{2p},p}(0,T;\L^p(\partial\Omega;T\Sigma))\cap L^{p}(0,T;\W^{1-\frac1p,p}(\partial\Omega)).
    $$
    Note that we consider functions in $T\Sigma$, which is a vector space.
	The spaces $Z_T,\widetilde{W}_T$ are endowed with the norms
	\begin{align}
		& \Vert f \Vert_{Z_T}:=\Vert f \Vert_{L^p(0,T;\W^{2,p}(\Omega))}+\Vert f\Vert_{W^{1,p}(0,T;\L^p(\Omega))}+\Vert f(0) \Vert_{\W^{2-\frac2p,p}(\Omega)}\\
        	& \Vert f \Vert_{\widetilde W_T}:=\Vert f \Vert_{L^{p}(0,T;\W^{1-\frac1p,p}(\partial\Omega))}+\Vert f\Vert_{W^{\frac12-\frac1{2p},p}(0,T;\L^p(\partial\Omega))}+\Vert f(0) \Vert_{\W^{(1-\frac1p)(1-\frac{2}{p-1}),p}(\partial\Omega)},
	\end{align}
	and we recall that, as in \cite[Lemma 2]{Saal1}, by standard embeddings, there exists $C(T_0)>0$ such that
	\begin{align}\label{BUC}
		\norm{\u}_{BUC([0,T];\W^{2-\frac2p,p}(\Omega))}\leq C(T_0)\norm{\u}_{X_T},
	\end{align}
    for any $T\in(0,T_0)$. It also holds (see, e.g., \cite{15H})
    \begin{align*}
    L^{p}(0,T;\W^{1-\frac1p,p}(\partial\Omega))\cap W^{\frac12-\frac1{2p},p}(0,T;\L^p(\partial\Omega))\hookrightarrow BUC([0,T];\W^{(1-\frac1p)(1-\frac{2}{p-1}),p}(\partial\Omega)), 
    \end{align*}
    which explains the choice of the norm for the space $\widetilde W_T$. Furthermore, it holds (see, e.g., \cite{15H,H})
    \begin{align}
    Tr_{\partial\Omega} (\nabla u),\  Tr_{\partial\Omega} (u)\in  L^{p}(0,T;\W^{1-\frac1p,p}(\partial\Omega))\cap W^{\frac12-\frac1{2p},p}(0,T;\L^p(\partial\Omega)),\quad \forall v\in Z_T. 
    \label{embf}
    \end{align}
In conclusion, we also have the following interpolation and embedding result
\begin{align*}
    (\W^{2-\frac2p,p}(\Omega),\bL^p(\Omega))_{\theta,p}\hookrightarrow  \bW^{(1-\theta)(2-\frac2p),p}(\Omega),
\end{align*}
where $\theta\in(0,1)$ is sufficiently close to 1, so that $ W^{(1-\theta)(2-\frac2p),p}(\Omega)    \hookrightarrow C^{1,\alpha}(\overline\Omega)$ for some $\alpha\in(0,1)$.
Now, since $Z_T\hookrightarrow W^{1,p}(0,\tT;\bL^p(\Omega)\hookrightarrow C^{0,1-\frac 1p}([0,\tT];\bL^p(\Omega))$, by \cite[Lemma 1]{AWe}, we also infer, by the interpolation result above \eqref{BUC},
\begin{align}
    \label{holder}
    Z_T\hookrightarrow C^{0,\theta(1-\frac1p)}([0,\tT];\W^{(1-\theta)(2-\frac2p),p}(\Omega) )\hookrightarrow C^{0,\gamma}([0,\tT];\mathbf C^{1,\alpha}(\overline\Omega)),
\end{align}
where $\gamma:=\theta(1-\frac1p)$.
	We also set $\YT:=L^p(0,T;\L^p(\Omega))$  and
          we introduce the spaces 
	$$
	\XT:=\{u\in Z_T:\ u(0)=\bP u_0\}, \qquad W_T:=\{v\in \widetilde W_T:\ v(0)=0\}.
	$$
	We now consider a solution $\widetilde u=u+\frac 1N\mathbf e$, with $u\in T\Sigma$, linearize the system \eqref{eq1}-\eqref{eq3}, and write it as follows:
	\begin{align*}
		\mathcal{L}(v)=\mathcal{F}(\u).
	\end{align*}
 Here,  the  linear operator $\mathcal{L}:\XT\to\YT\times W_T$  is
defined  as 
	\begin{align}
		\label{L}
		\mathcal{L}(v):=\begin{cases}
        \partial_tv-((\normmm{u_0}^2Id-\bP u_0\otimes \bP u_0)\Delta v)+\omega v,\quad\text{ in }\Omega\\
        ((\normmm{u_0}^2Id-\bP u_0\otimes \bP u_0)\nabla v)\cdot \mathbf n,\quad\text{ on }\partial\Omega
        \end{cases}
	\end{align}
	for some arbitrary $\omega>0$, where $\bP$ is the projection onto $T\Sigma$ and we have used the fundamental identity $\bP((\normmm{u_0}^2Id-u_0\otimes u_0)\Delta v)=((\normmm{u_0}^2Id-\bP u_0\otimes \bP u_0)\Delta v)$, for any $v\in X_T$. 
        We also introduce the possibly nonlinear operator
        $\mathcal{F}=(\mathcal F_1,\mathcal F_2):\XT\to\YT\times W_T$  given by 
	\begin{align} 
	\non	\mathcal{F}_1(v)&:=\omega v-\bP((\normmm{u_0}^2Id-u_0\otimes u_0)\Delta v)\non\\&\quad \non+\bP(\Div((\vert{v+\frac1N\mathbf e}\vert^2Id-(v+\frac1N\mathbf e)\otimes (v+\frac1N\mathbf e))\nabla v))-v\normmm{\nabla v}^2\\&
   \quad  +(v\cdot \partial_kv)\partial_k v-\bP\partial_u\Psi(v+\frac1N\mathbf e),
    \label{F}
	\end{align}
    \begin{align}
   \mathcal F_2(v):= -((\vert{v+\frac1N\mathbf e}\vert^2Id-v\otimes v)\nabla v)\cdot \mathbf n+((\normmm{u_0}^2Id-\bP u_0\otimes \bP u_0)\nabla v)\cdot \mathbf n,\quad\text{ on }\partial\Omega.
    \label{F2}
    \end{align}
	We aim now at proving that, under the assumptions of Theorem \ref{thm1}, there is a constant $C(T, R)>0$ such that
	\begin{align}
		\Vert\mathcal{F}(v_1) -\mathcal{F}(v_2)\Vert_{\YT\times W_T}
		\leq C(T, R)\Vert v_1 - v_2\Vert_{\XT},
		\label{CR}\end{align}
	for all $v_i\in \XT$ with $\Vert v_i\Vert_{\XT}
	\leq R$, $R>0$, and $i = 1, 2$. Furthermore it holds
	$C(T, R) \to0$ as ${T} \to0$.
	
	To prove this, we fix $v_i\in \XT$ with $\Vert v _i\Vert_{\XT}
	\leq R$, $R>0$, and split $\mathcal F$ in its summands.
    Namely, we start by using \eqref{holder}
    \begin{align*}
    &\norm{-\bP((\normmm{u_0}^2Id-u_0\otimes u_0)\Delta (v_1-v_2))+\bP((\vert{v_1+\frac1N\mathbf e}\vert^2Id-(v_1+\frac1N\mathbf e)\otimes (v_1+\frac1N\mathbf e))\Delta v_1)\right.\\&\left.-\bP((\vert{v_2+\frac1N\mathbf e}\vert^2Id-(v_2+\frac1N\mathbf e)\otimes (v_2+\frac1N\mathbf e))\Delta v_2)}_{\YT}\\&
    = \norm{-((\normmm{u_0}^2Id-\bP u_0\otimes \bP u_0)\Delta (v_1-v_2))+(\vert{v_1+\frac1N\mathbf e}\vert^2Id- v_1\otimes  v_1)\Delta v_1)\right.\\&\left.\quad -(\vert{v_2+\frac1N\mathbf e}\vert^2Id- v_2\otimes  v_2)\Delta v_2)}_{\YT}\\&
    \leq \norm{(\normmm{u_0}^2-\vert{v_1+\frac1N\mathbf e}\vert^2)\Delta(v_1-v_2)}_{\YT}+\norm{(\bP u_0\otimes \bP u_0- v_1\otimes  v_1)\Delta (v_1-v_2)}_{\YT}\\&
    \quad +\norm{(\vert{v_1+\frac1N\mathbf e}\vert^2-\vert{v_2+\frac1N\mathbf e}\vert^2)\Delta v_2}_{\YT}+\norm{( v_1\otimes  v_1- v_2\otimes  v_2)\Delta v_2}_{\YT}\\&
    \leq C(R)\norm{u_0-(v_1+\frac1N\mathbf e)}_{L^\infty(0,T;\L^\infty(\Omega))}\norm{v_1-v_2}_{\XT}+C(R)\norm{\bP u_0-v_1}_{L^\infty(0,T;\L^\infty(\Omega))}\norm{v_1-v_2}_{\XT}\\&\quad  +
    C(R)\norm{v_1-v_2}_{L^\infty(0,T;\L^\infty(\Omega))}\norm{v_2}_{L^p(0,T;\W^{2,p}(\Omega))}\\&
    \leq 
    C(R)T^{\gamma}\norm{v_1}_{\mathbf C^{0,\gamma}(\overline\Omega\times[0,T])}\norm{v_1-v_2}_{\XT}+C(R)T^{\gamma}\norm{v_1-v_2}_{\mathbf {C}^{0,\gamma}(\overline\Omega\times[0,T])}\\&
    \leq C(R)T^\gamma \norm{v_1-v_2}_{\XT},
    \end{align*}
    where we used the fact that $v_i(0)+\frac1N\mathbf e=u_0$, $i=1,2$, and $v_1(0)-v_2(0)=0$. Proceeding in the estimates, recalling \eqref{BUC} and the embedding $W^{2-\frac2p,p}(\Omega)\hookrightarrow W^{1,2p}(\Omega)$, we have 
\begin{align*}
&\norm{\bP(\partial_k(\vert{v_1}+\frac1N\mathbf e\vert^2\delta_{ij}-(v_{1,i}+\frac1N\mathbf e_i)(v_{1,j}+\frac1N\mathbf e_j))\partial_kv_{1,j}\right.\\&\left.-\partial_k(\vert{v_2+\frac1N\mathbf e}\vert^2\delta_{ij}-(v_{2,i}+\frac1N\mathbf e_i)(v_{2,j}+\frac1N\mathbf e_j))\partial_kv_{2,j})_{i=1}^N}_{\YT}\\&
\leq C(R)T^{\frac1p}\norm{v_1-v_2}_{L^{\infty}(0,T;\L^\infty(\Omega))}(\norm{v_1}_{L^{\infty}(0,T;\W^{1,2p}(\Omega))}+\norm{v_2}_{L^{\infty}(0,T;\W^{1,2p}(\Omega))})\\&\times (\norm{v_1}_{L^\infty(0,T;\W^{1,2p}(\Omega))}+\norm{v_2}_{L^\infty(0,T;\W^{1,2p}(\Omega)})\\&+C(R)T^{\frac1p}\norm{v_1-v_2}_{L^{\infty}(0,T;\W^{1,2p}(\Omega))}(\norm{v_1}_{L^\infty(0,T;\W^{1,2p}(\Omega)}+\norm{v_2}_{L^\infty(0,T;\W^{1,2p}(\Omega))})\\&
\leq C(R)T^\frac1p \norm{v_1-v_2}_{\XT}.
\end{align*}
 Then, by completely analogous arguments, we have 
\begin{align*}
&\norm{(v_1+\frac1N\mathbf e)\normmm{\nabla v_1}^2+(v_1\cdot \partial_kv_1)\partial_k v_1-(v_2+\frac1N\mathbf e)\normmm{\nabla v_2}^2  -(v_2\cdot \partial_kv_2)\partial_k v_2}_{\YT}
\\&
\leq C(R)T^{\frac1p}\norm{v_1-v_2}_{L^{\infty}(0,T;\L^\infty(\Omega))}(\norm{v_1}_{L^{\infty}(0,T;\W^{1,2p}(\Omega))}+\norm{v_2}_{L^{\infty}(0,T;\W^{1,2p}(\Omega))})\\&\times (\norm{v_1}_{L^\infty(0,T;\W^{1,2p}(\Omega))}+\norm{v_2}_{L^\infty(0,T;\W^{1,2p}(\Omega)})\\&+C(R)T^{\frac1p}\norm{v_1-v_2}_{L^{\infty}(0,T;\W^{1,2p}(\Omega))}(\norm{v_1}_{L^\infty(0,T;\W^{1,2p}(\Omega)}+\norm{v_2}_{L^\infty(0,T;\W^{1,2p}(\Omega))})\\&
\leq C(R)T^\frac1p \norm{v_1-v_2}_{\XT}.
\end{align*}
                                   In conclusion, since $\partial_u\Psi$ is locally Lipschitz continuous, it is immediate to infer, again using \eqref{BUC}, that 
\begin{align*}
\norm{\bP\partial_u\Psi(v_1+\frac1N \mathbf e)-\bP\partial_u\Psi(v_2+\frac1N \mathbf e)}_{\YT}\leq C(R)T^\frac1p\norm{v_1-v_2}_{\XT}.
\end{align*}
Collecting all the estimates above, we have thus shown that there is a constant $C_1(T, R)>0$ such that
	\begin{align}
		\Vert\mathcal{F}_1(v_1) -\mathcal{F}_1(v_2)\Vert_{\YT}
		\leq C_1(T, R)\Vert v_1 - v_2\Vert_{X_T},
		\label{est1A}\end{align}
	for all $v_i\in \XT$ with $\Vert v_i\Vert_{\XT}
	\leq R$, $R>0$, and $i = 1, 2$, with $C_1(T, R) \to0$ as ${T} \to0$.

Concerning the boundary operator $\mathcal F_2$, we have, using the embedding $\bW^{2,p}(\Omega)\hookrightarrow\bW^{2-\frac1p,p}(\partial\Omega)$, as well as \eqref{holder}, and recalling that $v_i(0)=\bP u_0$, $i=1,2$,
\begin{align*}
&\norm{\normmm{v_1+\frac1N\mathbf e}^2\nabla v_1\cdot \mathbf n-(\normmm{u_0}^2\nabla v_1)\cdot \mathbf n-\normmm{v_2+\frac1N\mathbf e}^2\nabla v_2\cdot \mathbf n+(\normmm{u_0}^2\nabla v_2)\cdot \mathbf n}_{\Lps}\\&
\leq \norm{\Big(\normmm{u_0}^2-\normmm{v_1+\frac1N\mathbf e}^2\Big)\nabla (v_1-v_2)\cdot \mathbf n}_{\Lps}\\&
\quad +\norm{\Big(\normmm{v_1+\frac1N\mathbf e}^2-\normmm{v_2+\frac1N\mathbf e}^2\Big)\nabla v_2\cdot \mathbf n}_{\Lps}\\&
\leq C\norm{u_0+v_1+\frac1N\mathbf e}_{L^\infty(0,T;\mathbf C(\overline\Omega))}\norm{\bP u_0-v_1}_{L^\infty(0,T;\mathbf C(\overline\Omega))}\norm{v_1-v_2}_{L^p(0,T;\W^{2-\frac1p,p}(\partial\Omega))}\\&
\quad +C\norm{u_0+v_1+\frac1N\mathbf e}_{L^\infty(0,T;\mathbf C(\overline\Omega))}\norm{\bP u_0-v_1}_{L^p(0,T;\bW^{1-\frac1p}(\partial\Omega))}\norm{v_1-v_2}_{L^\infty(0,T;\mathbf C^1(\overline\Omega))}\\&
\quad +C\norm{u_0+v_1+\frac1N\mathbf e}_{L^p(0,T;\mathbf W^{1-\frac1p} (\partial\Omega))}\norm{\bP u_0-v_1}_{L^\infty(0,T;\mathbf C(\overline\Omega))}\norm{v_1-v_2}_{L^\infty(0,T;C^1(\overline{\Omega}))}\\&
\quad +C\norm{v_2+v_1+\frac2N\mathbf e}_{L^\infty(0,T;\mathbf C(\overline\Omega))}\norm{v_1-v_2}_{L^\infty(0,T;\mathbf C(\overline\Omega))}\norm{v_2}_{L^p(0,T;\W^{2-\frac1p,p}(\partial\Omega))}\\&\quad +C\norm{v_2+v_1+\frac2N\mathbf e}_{L^\infty(0,T;\mathbf C(\overline\Omega))}\norm{v_1-v_2}_{L^p(0,T;\mathbf W^{1-\frac1p}(\partial\Omega))}\norm{v_2}_{L^\infty(0,T;\mathbf C^1(\overline\Omega))}
\\&\quad +C\norm{v_2+v_1+\frac2N\mathbf e}_{L^\infty(0,T;\mathbf W^{1-\frac1p}(\partial\Omega))}\norm{v_1-v_2}_{L^\infty(0,T;\mathbf C(\overline\Omega))}\norm{v_2}_{L^p(0,T;\W^{1,p}(\partial\Omega))}\\&
\leq 
C(R)T^\gamma \norm{v_1-v_2}_{X_T}.
\end{align*}
Let us denote $W^1_T:=\Wps$ and $W_T^2:=\Lps$.
With the same arguments as above, which we omit for the sake of brevity, we then also have 
\begin{align*}
&\norm{-(v_1\otimes v_1)\nabla v_1\cdot \mathbf n+(\bP u_0\otimes \bP u_0)\nabla v_1\cdot \mathbf n+(v_2\otimes v_2)\nabla v_2\cdot \mathbf n-(\bP u_0\otimes \bP u_0)\nabla v_2\cdot \mathbf n}_{W_T^2}\\&\leq C(R)T^\gamma \norm{v_1-v_2}_{X_T}.
\end{align*}
Proceeding in the estimates, we then have, by using the embeddings \eqref{embf} and \eqref{holder},
\begin{align*}
&\norm{\normmm{v_1+\frac1N\mathbf e}^2\nabla v_1\cdot \mathbf n-(\normmm{u_0}^2\nabla v_1)\cdot \mathbf n-\normmm{v_2+\frac1N\mathbf e}^2\nabla v_2\cdot \mathbf n+(\normmm{u_0}^2\nabla v_2)\cdot \mathbf n}_{W^1_T}\\&
\leq \norm{\Big(\normmm{u_0}^2-\normmm{v_1+\frac1N\mathbf e}^2\Big)\nabla (v_1-v_2)\cdot \mathbf n}_{W^1_T}\\&
\quad +\norm{\Big(\normmm{v_1+\frac1N\mathbf e}^2-\normmm{v_2+\frac1N\mathbf e}^2\Big)\nabla v_2\cdot \mathbf n}_{W^1_T}\\&
\leq C\norm{u_0+v_1+\frac1N\mathbf e}_{L^\infty(0,T;\mathbf C(\overline\Omega))}\norm{\bP u_0-v_1}_{L^\infty(0,T;\mathbf C(\overline\Omega))}\norm{\nabla (v_1-v_2)}_{W^1_T}\\&
\quad +C\norm{u_0+v_1+\frac1N\mathbf e}_{L^\infty(0,T;\mathbf C(\overline\Omega))}\norm{\bP u_0-v_1}_{W^1_T}\norm{\nabla (v_1-v_2)}_{L^\infty(0,T;\mathbf C(\overline\Omega))}\\&
\quad +C\norm{u_0+v_1+\frac1N\mathbf e}_{W^1_T}\norm{\bP u_0-v_1}_{L^\infty(0,T;\mathbf C(\overline\Omega))}\norm{\nabla (v_1-v_2)}_{L^\infty(0,T;\mathbf C(\overline\Omega))}\\&
\quad +C\norm{v_2+v_1+\frac2N\mathbf e}_{L^\infty(0,T;\mathbf C(\overline\Omega))}\norm{v_1-v_2}_{L^\infty(0,T;\mathbf C(\overline\Omega))}\norm{\nabla v_2}_{W^1_T}\\&\quad +C\norm{v_2+v_1+\frac2N\mathbf e}_{L^\infty(0,T;\mathbf C(\overline\Omega))}\norm{v_1-v_2}_{W^1_T}\norm{\nabla v_2}_{L^\infty(0,T;\mathbf C(\overline\Omega))}
\\&\quad +C\norm{v_2+v_1+\frac2N\mathbf e}_{W^1_T}\norm{v_1-v_2}_{L^\infty(0,T;\mathbf C(\overline\Omega))}\norm{\nabla v_2}_{L^\infty(0,T;\mathbf C(\overline\Omega))}\\&
\leq 
C(R)T^\gamma \norm{v_1-v_2}_{X_T}.
\end{align*}
Again, by a very similar argument, we also infer 
\begin{align*}
&\hspace{-2em}\norm{-(v_1\otimes v_1)\nabla v_1\cdot \mathbf n+(\bP u_0\otimes \bP u_0)\nabla v_1\cdot \mathbf n+(v_2\otimes v_2)\nabla v_2\cdot \mathbf n-(\bP u_0\otimes \bP u_0)\nabla v_2\cdot \mathbf n}_{W^1_T}\\&\hspace{-2em}\leq C(R)T^\gamma \norm{v_1-v_2}_{X_T}.
\end{align*}
As a consequence, we have thus shown that there is a constant $C_2(T, R)>0$ such that
	\begin{align}
		\Vert\mathcal{F}_2(v_1) -\mathcal{F}_2(v_2)\Vert_{W_T}
		\leq C_2(T, R)\Vert v_1 - v_2\Vert_{X_T},
		\label{est1B}\end{align}
	for all $v_i\in \XT$ with $\Vert v_i\Vert_{\XT}
	\leq R$, $R>0$, and $i = 1, 2$, where $C_2(T, R) \to0$ as ${T} \to0$.
This, together with \eqref{est1A}, leads to \eqref{CR}.
	
	We now  concentrate  on the operator $\mathcal L$. The fact
        that, for any ${T}>0$, this  operator  is invertible from $\YT\times W_T$
        to $\XT$ is an immediate consequence of Theorem \ref{Maxreg}, with $\overline u=u_0\in \W^{2-\frac2p,p}(\Omega)\hookrightarrow \mathbf C^{1,l}(\overline\Omega)\hookrightarrow \mathbf C^{0,l}(\overline\Omega)\cap \W^{1-\frac1p,p}(\partial\Omega)$ for some $l\in(0,1)$, since $p>d+2$. Indeed, the 
        theorem can be easily shown to hold also on bounded intervals $(0,T)$.  As a consequence, by the 
        bounded inverse
        theorem, for any $T>0$ 
        there exists $C(T)>0$, possibly depending on $T$, such that 
	$$
	\norm{\mathcal{L}^{-1}}_{\mathcal{L}(\YT\times W_T,\XT)}\leq C(T),\quad \forall T>0.
	$$
	On the other hand, the constant above does not
      change with $T$, as it can be seen from a trivial extension argument (see for instance the proof of \cite[Lemma 7]{AWe}), leading to the existence of $C(T_0)$  such that
	\begin{align}
	\norm{\mathcal{L}^{-1}}_{\mathcal{L}(\YT\times W_T,\XT)} \leq  C(T_0),\quad \forall 0<T< T_0.
		\label{map}\end{align}
	
	We can now complete the existence proof.
    We aim at
        solving via a fixed-point  argument the equation 
	$$
	\u=\mathcal{L}^{-1}\mathcal{F}\u \quad\text{in }\XT,
	$$
    {for some }
$T\in(0,T_0)$.	Consider a generic $\overline{u}\in \XT$, and then we fix $R>0$ such that 
	$\overline{u}\in\overline{B}_{ R}^{X_{T_0}}(0)$, where
        $\overline{B}_R^{\XT}(0)$ is the closed ball of $\XT$ of
        radius $R$ centered at  $0$.   Then we fix $0<T< T_0$ (possibly depending on $R$) such that the operator $\mathcal{L}^{-1}\mathcal{F}$ is a $(1/4)$-contraction mapping from $\XT$ to $\XT$. We can do this thanks to \eqref{CR} and \eqref{map}, since 
	\begin{align*}
		\Vert\mathcal{L}^{-1}\mathcal{F}(u_1) -\mathcal{L}^{-1}\mathcal{F}(u_2)\Vert_{\XT}
		\leq \norm{\mathcal{L}^{-1}}_{\mathcal{L}(\YT\times W_T,\XT)}C(T, R)\Vert u_1 - u_2\Vert_{\XT}\leq C(T_0) C(T, R)\Vert u_1 - u_2\Vert_{\XT},
	\end{align*}
	and thus we choose $T$ sufficiently small so that $C(T_0) C(T, R)\leq \frac14$. Then $\mathcal{L}^{-1}\mathcal{F}$ is well defined from $\overline{B}_R^{\XT}(0)$ to itself, since for any $\v\in \overline{B}_R^{\XT}(0)$ we have
	\begin{align*}
		\norm{\mathcal{L}^{-1}\mathcal{F}u}_{\XT}&\leq 	\norm{\mathcal{L}^{-1}\mathcal{F}u-\mathcal{L}^{-1}\mathcal{F}\overline{u}}_{\XT}+	\norm{\mathcal{L}^{-1}\mathcal{F}\overline{u}}_{\XT}\\&
		\leq
          \frac14\norm{u-\overline{u}}_{\XT}+\norm{\mathcal{L}^{-1}\mathcal{F}\overline{u}}_{X_{T_0}}
          < R,
	\end{align*}
	since $u,\overline{u}\in \overline{B}_R^{\XT}(0)$.
	
	Therefore, by Banach fixed point theorem applied on $\mathcal
        {L}^{-1}\mathcal F: \overline{B}_R^{\XT}(0)\to
        \overline{B}_R^{\XT}(0)$, there exists a unique solution
        $\u\in \overline{B}_R^{\XT}(0)\subset \XT$ to the problem
        under study. By a standard argument it is also easy to show
        that the solution  $\u\in\overline{B}_R^{\XT}(0)\subset\XT$ we
        just found is unique in $\XT$ (see, e.g. \cite[Proof of Theorem 3.1, Section 4.1.1]{PU}. Then, the unique solution to \eqref{eq1}-\eqref{eq3}, with the same regularity properties as $u\in \overline{B}_R^{\XT}(0)$, is $\widetilde u=u+\frac1N\mathbf e$. This concludes the proof of the theorem.
\end{proof}

\section{Proof of Theorem \ref{thm:existenceweak}. Existence of global weak solutions}\label{exweak}
In order to prove the existence of global-in-time weak solutions, we  consider the following minimizing movement scheme (\cite{AGS,degiorgi}): given $u_0\in \mathbf G$, for $h>0$, solve
  \begin{align}
    \nonumber u^{0}&=u_0,\\
   u^{n+1}&:=\argmin{u\in H^1(\Omega;\Sigma)} \widehat\cE_1(u)+\int_\Omega \Psi(u)\dx +\frac{\norm{u-u^{n}}_{\bL^2(\Omega)}^2}{2h}, \quad n\geq 0.\label{min1b}
\end{align}
Additionally, we consider the De Giorgi interpolation (cf. \cite{AGS,degiorgi}) as follows:  
  \begin{align}
\non u^h(nh)&:=u^n,
\\u^h(t)&:=\argmin{u\in H^1(\Omega;\Sigma)} \widehat\cE_1(u)+\int_\Omega \Psi(u)\dx +\frac{\norm{u-u^{n}}_{\bL^2(\Omega)}^2}{2(t-nh)}, \quad n\geq 0,\quad t\in (nh,(n+1)h].\label{min1c}
\end{align}
We will prove the existence of global weak solutions to \eqref{eq1}-\eqref{eq3} passing to the limit as $h\to0$.
Observe that the problem admits solutions by means of the Direct Method of Calculus of Variations, for $h>0$ sufficiently small. This is possible as $\Psi_1$ is convex and $\Psi_2$ is quadratic, and the functional $\widehat \cE_1(\cdot)$ is sequentially lower semicontinuous with respect to the weak $\bH^1(\Omega)$ topology, since, given a sequence $v_k\rightharpoonup v$ weakly in $\bH^1(\Omega)$ (and thus strongly in $\bL^2(\Omega)$) we have, for any $i,j=1,\ldots,N$,
\begin{align*}
\widehat \cE_1(v)\geq \sum_{i,j=1}^N\int_\Omega (v_{k,i}\nabla v_{k,j}-v_{k,j}\nabla v_{k,i})\cdot \xi_{ij}\dx-\frac12 \norm{\xi_{ij}}^2_{L^2(\Omega)},\quad \forall \xi_{ij}\in   C^\infty_c(\Omega),
\end{align*}
which is a trivial consequence of Young's inequality, since
$$
\frac12\normmm{v_{k,i}\nabla v_{k,j}-v_{k,j}\nabla v_{k,i}}^2\geq (v_{k,i}\nabla v_{k,j}-v_{k,j}\nabla v_{k,i})\cdot \xi_{ij}-\frac12 \normmm{\xi_{ij}}^2,
$$
for any $i,j=1,\dots,N$.
Therefore, taking the liminf as $k\to\infty$ we infer 
\begin{align*}
&\liminf_{k\to\infty} \frac12\int_\Omega\normmm{v_{k,i}\nabla v_{k,j}-v_{k,j}\nabla v_{k,i}}^2\dx \\&\geq \int_\Omega (v_{i}\nabla v_{j}-v_{j}\nabla v_{i})\cdot \xi_{ij}\dx-\frac12 \norm{\xi_{ij}}^2_{L^2(\Omega)},\quad \forall \xi_{ij}\in C^\infty_c(\Omega),
\end{align*}
so that, taking the supremum over $\xi_{ij}\in C_c^\infty(\Omega)$, we obtain, for any $i,j=1,\ldots,N$,
\begin{align}
\liminf_{k\to\infty}\frac12\int_\Omega\normmm{v_{k,i}\nabla v_{k,j}-v_{k,j}\nabla v_{k,i}}^2\dx\geq \frac12\int_\Omega\normmm{v_i\nabla v_{j}-v_{j}\nabla v_{i}}^2\dx.
    \label{lsc0}
\end{align} 
As a consequence, taking the sum over $i,j=1,\ldots,N$, we finally infer
\begin{align}
\liminf_{k\to\infty}\widehat\cE_1(v_k)\geq \sum_{i,j=1}^N \liminf_{k\to\infty}\frac12\int_\Omega\normmm{v_{k,i}\nabla v_{k,j}-v_{k,j}\nabla v_{k,i}}^2\dx\geq \widehat\cE_1(v).
    \label{lsc}
\end{align}
Let us then consider the sequences $\{u^n\}_{n\in\N}$ and the corresponding De Giorgi's interpolant $u^h$, where we 
point out here that $t\mapsto u_h(t)$ can be chosen to be Bochner measurable with values, for instance, in $\bL^2(\Omega)$, applying a suitable measurable selection theorem. We refer to \cite[Appendix B]{AbelsPoiatti} for a similar argument.

  {The proof of Theorem~\ref{thm:existenceweak} is now structured in the following sections as follows. First, we prove uniform pointwise bounds on our approximations, heavily exploiting the fundamental property that $\widehat \cE_1(v)=\frac1N\norm{\nabla \bP v}_{\bL^2(\Omega)}^2+\widehat \cE_1(\bP v)$. Then, we derive the sharp energy-dissipation inequality for the approximations. To get a first compactness result, we extract uniform estimates from this inequality. 
Then we improve this by proving higher integrability of the gradients which is crucial to prove the strong convergence of the gradients. 
This finally allows us to pass to the limit in the weak formulation and in the energy-dissipation inequality.}

\subsection{Boundedness of $\{u^n\}_n$ and $u^h$}
The first key result is the fact that the sequences are bounded, under assumption \eqref{assbasic} on the potential $\Psi$. More precisely, we can prove that, if $\normmm{\bP u_{n-1}}\leq M$ for some $M>0$, then $\normmm{\bP u_n}\leq M$ and  $\normmm{\bP u^h(t)}\leq M$, for $t\in [(n-1)h,nh]$.  Note that at this level we do not actually need any upper bound on $M>0$, which will be needed only later on.
Now, note that, for $u:\Omega\to \Sigma$, 
\begin{align}
  \non  &\widehat\cE_1(u)=\sum_{i=1}^d \int_\Omega E_1( u,\partial_{x_i} u)\dx=\sum_{i=1}^d \int_\Omega \normmm{ u\wedge \partial_{x_i} u}^2\dx\\&
    =\sum_{i=1}^d \int_\Omega \frac1N\normmm{\partial_{x_i}u}^2+\normmm{ \bP u\wedge \partial_{x_i} u}^2\dx=\sum_{i=1}^d \int_\Omega \frac1N\normmm{\partial_{x_i}\bP u}^2+\normmm{ \bP u\wedge \partial_{x_i}\bP  u}^2\dx.\label{decomp}
\end{align}
Then, we see that, recalling the definition of the operator $T_M$ in \eqref{TM}, for any $M>0$, if $\normmm{\bP u(x)}>M$, then 
\begin{align*}
&\frac1N\normmm{\partial_{x_i}T_M(\bP u)}^2+\normmm{ T_M(\bP u)\wedge \partial_{x_i}T_M(\bP  u)}\\&=\frac1N\frac{M^2}{\normmm{\bP u}^2} \normmm{(Id-\frac{\bP u}{\normmm{\bP u}}\otimes \frac{\bP u}{\normmm{\bP u}})\partial_{x_i}\bP u}^2+\frac{M^2}{\normmm{\bP u}^2}\normmm{\bP u\wedge \partial_{x_i}\bP u}^2\\&\leq \frac1 N \normmm{\partial_{x_i}\bP u}^2+\normmm{\bP u\wedge \partial_{x_i}\bP u}^2,
\end{align*}
so that we can deduce 
\begin{align}
\widehat\cE(T_M(u)+\tfrac1N\mathbf e)\leq \widehat\cE(u),\quad \forall u\in H^1(\Omega;\Sigma).
\end{align}
In conclusion, if $\normmm{\bP u^n}\leq M$, then it holds 
\begin{align*}
\normmm{T_M(u)+\tfrac 1N\mathbf e-u^{n}}=\normmm{T_M(u)- \bP u^{n}}\leq \normmm{\bP u- \bP u^{n}}=\normmm{ u-  u^{n}},
\end{align*}
so that 
\begin{align*}
\norm{T_M(u)+\tfrac 1N\mathbf e-u^{n}}_{\bL^2(\Omega)}\leq \norm{u-u^{n}}_{\bL^2(\Omega)}.
\end{align*}
As a consequence, if $u$ is optimal for \eqref{min1b}, then also $T_M(u)+\tfrac1N\mathbf e$ is optimal and, of course, bounded by $M>0$ in its projection on $T\Sigma$. By assumption we have $u_0\in\mathbf G$, and  $\normmm{\bP u_0}\leq  M$ (which is not a constraint for $N\leq 4$, recallig Remark \ref{smallness}), and thus starting from $n=0$ we can construct a sequence $\{u^n\}_n$ such that 
\begin{align}
\normmm{\bP u^n}\leq M,\quad \forall n\in \N,
\label{unb}
\end{align}
for $M^2<\frac 4N$. Analogously, we can construct the De Giorgi interpolant $u^h$, which of course enjoys the same property, namely 
\begin{align}
\normmm{\bP u^h(t)}\leq M,\quad \forall t\geq0.
\label{unb1}
\end{align}
\subsection{Sharp energy inequality}
Since the energy is not jointly convex (cf. Remark \ref{noconvex}), we need to resort to De Giorgi's interpolation to obtain the desired sharp energy dissipation inequality. In particular, it is standard to verify (see, e.g., \cite{AGS}) that, for, e.g., $t\in(0,h]$, the functional 
\begin{align*}
f_h(t):=\widehat\cE_1(u^h(t))+\int_\Omega \Psi(u^h(t))\dx +\frac{\norm{u^h(t)-u_{0}}_{\bL^2(\Omega)}^2}{2t}
\end{align*}
is locally Lipschitz continuous and satisfies
\begin{align}
    \ddt f_h(t)=-\frac1 {2t^2}\norm{u^h(t)-u_{0}}_{\bL^2(\Omega)}^2,\quad \text{ for a.e.\ }t\in(0,h].
\label{derivation}
\end{align}
We can also see that $t\mapsto \widehat\cE_1(u^h(t))+\int_\Omega \Psi(u^h(t))\dx $ is monotone nonincreasing. 
Indeed, for $0<s<t$, we first obtain that $t\mapsto {\norm{u^h(t)-u_{0}}_{\bL^2(\Omega)}^2}$ is nondecreasing. 
To show this, we just exploit the optimality properties of $u^h(t)$ and $u^h(s)$ to deduce, for $0<s<t$,
\begin{align*}
&\widehat\cE_1(u^h(t))+\int_\Omega \Psi(u^h(t))\dx +\frac{\norm{u^h(t)-u_{0}}_{\bL^2(\Omega)}^2}{2t}\\&
\leq \widehat\cE_1(u^h(s))+\int_\Omega \Psi(u^h(s))\dx +\frac{\norm{u^h(s)-u_{0}}_{\bL^2(\Omega)}^2}{2t}\pm \frac{\norm{u^h(s)-u_{0}}_{\bL^2(\Omega)}^2}{2s}\\&
\leq \widehat\cE_1(u^h(t))+\int_\Omega \Psi(u^h(t))\dx +\frac{\norm{u^h(t)-u_{0}}_{\bL^2(\Omega)}^2}{2s}+\left(\frac 1{2t}-\frac1{2s}\right)\norm{u^h(s)-u_{0}}_{\bL^2(\Omega)}^2,
\end{align*}
so that, rearranging the terms, and recalling that $\tfrac 1 s>\tfrac 1 t$, we infer the desired monotonicity: 
\begin{align}\label{monot}
\norm{u^h(s)-u_{0}}_{\bL^2(\Omega)}^2\leq \norm{u^h(t)-u_{0}}_{\bL^2(\Omega)}^2,\quad \forall 0<s<t.
 \end{align}
The monotonicity of the energy is then a trivial consequence, since by the optimality properties of $u^h(t)$ we infer 
\begin{align*}
&\widehat\cE_1(u^h(t))+\int_\Omega \Psi(u^h(t))\dx +\frac{\norm{u^h(t)-u_{0}}_{\bL^2(\Omega)}^2}{2t}\\&
\leq \widehat\cE_1(u^h(s))+\int_\Omega \Psi(u^h(s))\dx +\frac{\norm{u^h(s)-u_{0}}_{\bL^2(\Omega)}^2}{2t}\\&
\leq \widehat\cE_1(u^h(s))+\int_\Omega \Psi(u^h(s))\dx +\frac{\norm{u^h(t)-u_{0}}_{\bL^2(\Omega)}^2}{2t},
\end{align*}
where in the last inequality we used the monotonicity \eqref{monot}. This gives 
\begin{align*}
&\widehat\cE_1(u^h(t))+\int_\Omega \Psi(u^h(t))\dx 
\leq \widehat\cE_1(u^h(s))+\int_\Omega \Psi(u^h(s))\dx ,\quad \forall 0<s<t,\quad t\in(0,h],
\end{align*}
as desired. Also, we observe that again the optimality property of $u^h(t)$ gives
\begin{align*}
&\widehat\cE_1(u^h(t))+\int_\Omega \Psi(u^h(t))\dx 
\leq \widehat\cE_1(u_0)+\int_\Omega \Psi(u_0)\dx ,\quad \forall  t\in(0,h].
\end{align*}
As a consequence, repeating these very same arguments on any interval $(nh,(n+1)h]$ we get
\begin{align*}
&\widehat\cE_1(u^h(t))+\int_\Omega \Psi(u^h(t))\dx 
\leq \widehat\cE_1(u^h(s))+\int_\Omega \Psi(u^h(s))\dx ,\quad \forall \ nh<s<t\leq (n+1)h,\\&
\widehat\cE_1(u^h(t))+\int_\Omega \Psi(u^h(t))\dx 
\leq \cE_1(u^{n},u^{n})+\int_\Omega \Psi(u^n)\dx ,\quad \forall  t\in(nh,(n+1)h],
\end{align*}
which finally give
\begin{align}
t\mapsto \widehat\cE_1(u^h(t))+\int_\Omega \Psi(u^h(t))\dx\quad\text{ is monotone nonincreasing }\forall t\geq 0,
\label{monot1}
\end{align}
entailing
\begin{align}
\label{ua}
\sup_{t\geq 0}\left(\widehat\cE_1(u^h(t))+\int_\Omega \Psi(u^h(t))\dx\right)\leq \cE_1(u_0,u_0)+\int_\Omega \Psi(u_0)\dx, 
\end{align}
so that, since $\Psi\geq -C$ by assumption, 
\begin{align}
\label{ub}
\sup_{t\geq 0}\widehat\cE_1(u^h(t))\leq \widehat\cE_1(u_0)+\int_\Omega \Psi(u_0)\dx+C. 
\end{align}
To obtain the sharp energy inequality, we now integrate \eqref{derivation} in time over $(0,h]$, obtaining 
\begin{align}
&\non \widehat\cE_1(u^{1})+\int_\Omega \Psi(u^{1})\dx+\frac{1}{2h}\norm{u^{1}-u_{0}}_{\bL^2(\Omega)}^2+\int_0^h \frac{1}{2t^2}\norm{u^h(t)-u_0}^2_{\bL^2(\Omega)}\dt\\&\leq \widehat\cE_1(u_0)+\int_\Omega \Psi(u_0)\dx.
\label{energy ineq1}
\end{align}
We can now iterate this argument over the intervals $(nh,(n+1)h]$, so that, after a telescoping argument, the energy inequality becomes, for any $n\geq k$, $k\in \N\cup \{0\}$,
\begin{align}
&\widehat\cE_1(u^{n+1})+\int_\Omega \Psi(u^{n+1})\dx+\frac{1}{2h}\sum_{m=k+1}^{n+1}\norm{u^{m}-u^{m-1}}_{\bL^2(\Omega)}^2\non\\&+\sum_{m=k+1}^{n+1}\int_{(m-1)h}^{mh} \frac{1}{2(t-(m-1)h)^2}\norm{u^h(t)-u_{m-1}}^2_{\bL^2(\Omega)}\dt\non\\&\leq \widehat\cE_1(u^k)+\int_\Omega \Psi(u^k)\dx.
\label{energy ineq}
\end{align}
 Also, recalling \eqref{min1b}, the sequence $\{u_n\}_n$ satisfies, 
for any $w\in \bH^1(\Omega;T\Sigma)\cap\bL^\infty(\Omega)$,
\begin{align}
&\int_\Omega \frac1h(u^{n+1}-u^n)\cdot w\dx+\langle\delta \mathcal E_1(\cdot,u^{n+1})(u^{n+1})+\delta \mathcal E_1(u^{n+1},\cdot)(u^{n+1}),w\rangle\non\\&+\int_\Omega\Psi'(u^{n+1})\cdot w\dx=0 ,\quad \text{for any }n\in\N,\label{inclusion3}
\end{align}
which can be written more explicitly as 
\begin{align}
&\int_\Omega \frac1h(u^{n+1}-u^n)\cdot w\dx+\sum_{i=1}^N\int_\Omega (w\wedge \partial_{x_i}u^{n+1},u^{n+1}\wedge   \partial_{x_i}u^{n+1})\non\dx\\&+\sum_{i=1}^N\int_\Omega (u^{n+1}\wedge \partial_{x_i}w,u^{n+1}\wedge   \partial_{x_i}u^{n+1})\dx+\int_\Omega\Psi'(u^{n+1})\cdot w\dx=0 ,\quad \text{for any }n\in\N.\label{inclusion4}
\end{align}
Similarly, the De Giorgi interpolant $u^h$ satisfies, 
for any $w\in \bH^1(\Omega;T\Sigma)\cap\bL^\infty(\Omega)$,
\begin{align}
&\int_\Omega \frac 1{t-nh}(u^h(t)-u^{n})\cdot w\dx+\langle\delta \mathcal E_1(\cdot,u^h(t))(u^h(t))+\delta \mathcal E_1(u^h(t),\cdot)(u^h(t)),w\rangle\non\\&+\int_\Omega\Psi'(u^h(t))\cdot w\dx=0, \label{Degiorgi}
\end{align}
for any $n\in \N$ and $t\in(nh,(n+1)h]$.
\subsection{Uniform estimates}
Let us define $u_{h}$ the piecewise constant function corresponding to $\{u^{n+1}\}_n$ (this means that $u_{h}(t)=u^n$ for $t\in\left[nh,(n+1)h\right)$). Also, $\widehat u_{h}$ is the piecewise affine function corresponding to the same sequence, so that $\pt\widehat u_{h}=\frac{u^{n+1}-u^n}h$ for $t\in[nh,(n+1)h)$, $n\in\N$. On the other hand, as in \cite{AbelsPoiatti} we set
$$
g_h(t):=t-\floor{\frac th}h,\quad\text{for any }t\in(nh,(n+1)h),\quad \forall n\in \N,
$$
together with 
$$
v^h(t):=\frac{1}{g_h(t)}(v^h(t)-u_h(t)).
$$
 Thanks to \eqref{ub} and \eqref{energy ineq} (with $k=0$), we infer, for some $C>0$,
\begin{align}
&\esssup_{t\geq 0}\norm{\nabla u_{h}(t)}+\esssup_{t\geq 0}\norm{\nabla \widehat u_{h}(t)}+ \esssup_{t\geq 0}\norm{\nabla u^{h}(t)}\leq C,\label{estimates1}
\\&
\norm{\pt \widehat u_{h}}_{L^2(0,T;\bL^2(\Omega))}+\norm{ v_{h}}_{L^2(0,T;\bL^2(\Omega))}\leq C(T),\quad \forall T>0.
\label{estimates3}
\end{align}
Also, recalling \eqref{unb}-\eqref{unb1}, we have
\begin{align}
\normmm{\bP u_{h}}\leq M, \quad \normmm{\bP \widehat u_h}\leq M,\quad \vert{\bP u^{h}}\vert\leq M,\quad\text{ a.e. in }\Omega,
    \label{bounds0}
\end{align}
so that 
\begin{align}
\normmm{u_{h}}\leq \sqrt{M^2+\frac1{N}}, \quad \normmm{\widehat u_h}\leq \sqrt{M^2+\frac1{N}},\quad \vert{u^{h}}\vert\leq \sqrt{M^2+\frac1{N}},\quad\text{ a.e. in }\Omega,
    \label{bounds}
\end{align}
for some $C>0$. Moreover, $u_h$ and $\widehat u_h$ satisfy, from \eqref{inclusion4},
\begin{align}
&\int_\Omega \partial_t \widehat u_h\cdot w\dx+\sum_{i=1}^N\int_\Omega (w\wedge \partial_{x_i}u_h,u_h\wedge   \partial_{x_i}u_h)\non\dx\\&+\sum_{i=1}^N\int_\Omega (u_h\wedge \partial_{x_i}w,u_h\wedge   \partial_{x_i}u_h)\dx+\int_\Omega\Psi'(u_h)\cdot w\dx=0 ,\quad \text{for any }t\geq 0,\label{inclusion5}
\end{align}
whereas the De Giorgi interpolant $u^h$ and $v^h$ satisfy
\begin{align}
&\int_\Omega  v^h\cdot w\dx+\sum_{i=1}^N\int_\Omega (w\wedge \partial_{x_i}u^h,u^h\wedge   \partial_{x_i}u^h)\non\dx\\&+\sum_{i=1}^N\int_\Omega (u^h\wedge \partial_{x_i}w,u^h\wedge   \partial_{x_i}u^h)\dx+\int_\Omega\Psi'(u^h)\cdot w\dx=0,\label{inclusion5b}
\end{align}
for almost any $t\geq0$.
Then, up to subsequences, we have, also using the Aubin-Lions Theorem,
\begin{align}
    \label{c1ab}u_{h}\overset{*}{\rightharpoonup} u 
    &\qquad \text{weakly* in }\bL^\infty(\Omega\times(0,T);\Sigma),
    \\
    u_{h}\overset{*}{\rightharpoonup} u
    &\qquad \text{weakly* in }L^\infty(0,T;\bH^1(\Omega)),\label{weakH1}
    \\
    \label{DeGiorgiA}
    u^{h}\overset{*}{\rightharpoonup} z
    &\qquad \text{weakly* in }\bL^\infty(\Omega\times(0,T);\Sigma),
    \\
    \label{DeGiorgiB}
    u^{h}\overset{*}{\rightharpoonup} z
    &\qquad \text{weakly* in }L^\infty(0,T;\bH^1(\Omega)),
    \\
    \widehat u_h\to w
    &\qquad \text{strongly in }L^2(0,T;\bL^2(\Omega)),\label{strong1}
    \\
    \label{c2ab}\widehat u_h\overset{*}{\rightharpoonup} w
    &\qquad \text{weakly* in }\bL^\infty(\Omega\times(0,T);\Sigma),
    \\
    \partial_t\widehat u_h\rightharpoonup \partial_t w
    &\qquad \text{weakly in }L^2(0,T;\bL^2(\Omega;T\Sigma)),\label{c3a}
    \\
   v^h\rightharpoonup v
    &\qquad \text{weakly in }L^2(0,T;\bL^2(\Omega;T\Sigma)),\label{c3b}
\end{align}
for any $T>0$.
In order to identify $u$ and $w$, it is enough to note that, for $t\in[nh,(n+1)h)$, $t\leq T$, $T>0$, by \eqref{estimates3},
\begin{align*}
   \norm{u_h(t)-\widehat u_h(t)}_{\bL^2(\Omega)}&=  \norm{\widehat{u}_h(nh)-\widehat u_h(t)}_{\bL^2(\Omega)}\\&\leq \left(\int_{nh}^t\norm{\partial_t\widehat u_h(s)}_{\bL^2(\Omega)}^2\d s\right)^\frac12\sqrt h\\&\leq C(T)\sqrt h\to 0,
\end{align*}
as $h\to 0$, which gives, by Lebesgue's Dominated Convergence,
\begin{align*}
\int_0^T\norm{u_h(t)-\widehat u_h(t)}_{\bL^2(\Omega)}\dt\to 0,
\end{align*}
for any $T>0$. This, together with \eqref{c1ab} and \eqref{strong1}, gives $u=w$ almost everywhere in $\Omega\times(0,\infty)$.
 Of course, recalling \eqref{bounds}, this also entails 
\begin{align}
u_h\to u\quad \text{ strongly in }L^q(0,T;\bL^s(\Omega)),\quad \forall q,s\in[1,\infty),
    \label{uh}
\end{align}
as well as, from this convergence and \eqref{bounds}, 
\begin{align}
\normmm{u}\leq \sqrt{M^2+\frac 1N},\quad\text{ a.e. in }\Omega\times(0,\infty).
\label{boundu}
\end{align}
In conclusion, we need to identify $u$ with $z$, which is the limit of the De Giorgi interpolant $u^h$. To this aim, we note that, from \eqref{energy ineq}, we have
\begin{align*}
	\int_0^T\norm{u^h(t)-u_h(t)}^2_{\bL^2(\Omega))}\dt\leq \int_0^T\norm{v^h}_{\bL^2(\Omega)}^2g_h^2\dt\leq h^2 \int_0^T\norm{v^h}_{\bL^2(\Omega)}^2\dt,
\end{align*}
where we used that $\normmm{g_h(t)}\leq h$. This, together with \eqref{uh}, allows to infer that
\begin{align}
	u^h\to u\quad \text{ strongly in }L^2(0,T;\bL^2(\Omega)),
	\label{uhA}
\end{align}
which gives $z=u$ almost everyhwere in $\Omega\times(0,\infty)$, and, using the uniform bound \eqref{bounds0},
\begin{align}
	u^h\to u\quad \text{ strongly in }L^q(0,T;\bL^s(\Omega)),\quad \forall q,s\in[1,\infty).
	\label{uhB}
\end{align}

Note that to be precise also $v$ is not yet identified, but we will see that $v=\pt u$ in the last step of the proof.  
\subsection{  {Higher} integrability of the gradients}
In order to pass to the limit in  \eqref{inclusion5}, we now need to perform an elliptic argument to obtain higher integrability for $\nabla u_h$, which so far would only   {be controlled} in $L^\infty(0,\infty;\bL^2(\Omega))$. We prove the following general lemma for   {solutions} to a suitable elliptic PDE:
\begin{lemma}
\label{higher}
Let $\Omega$ be a bounded domain with Lipschitz boundary. Let then $u\in \bL^\infty(\Omega)\cap \bH^1(\Omega;\Sigma)$, $\normmm{\bP u}\leq M$ with $M>0$, be a solution to 
\begin{align}
    \sum_{i=1}^N(u\wedge \partial_{x_i}u,w\wedge \partial_{x_i} u)+\sum_{i=1}^N(u\wedge \partial_{x_i}u,u\wedge \partial_{x_i} w)=(g,w),\quad \forall w\in \bH^1(\Omega;T\Sigma)\cap \bL^\infty(\Omega),\label{weakform}
\end{align}
where $g\in \bL^2(\Omega)$.
Then, if $M<\sqrt{\frac 4N}$, there exists $p>2 $, independent of $u$, such that 
\begin{align}
\norm{u}_{W^{1,p}(\Omega)}\leq C(M)(1+\norm{u}_{\bH^1(\Omega)}+\norm{g}_{\bL^2(\Omega)}).\label{higherint}
\end{align}
\end{lemma}
\begin{remark}
	The smallness constraint on $M$ is apparently unavoidable, see also \cite[Theorem 2.3]{Giaquinta}. Nevertheless, here the smallness condition is larger than in \cite{Giaquinta}, as we are exploiting some further monotonicity property of the operators involved.
\end{remark}
\begin{proof}
The proof is based on a Caccioppoli inequality, a reverse Hölder inequality and a perturbation argument from \cite{GiaquintaModica}. A similar argument, applied to a different system with a similar constraint $u\in \Sigma$, can be found in \cite[Theorem 4.1]{Garcke}.
The sketch of the proof is as follows. We begin from interior regularity. For fixed $x_0\in \Omega$, let $R>0$ be such that 
\begin{align*}
    Q_{2R}(x_0):=\{x\in \R^d:\ \normmm{x_i-x_{0i}}<2R,\ i=1,\ldots, d \}\subset \Omega.
\end{align*}
We then introduce a smooth cutoff $\xi\in C^\infty_c(\Omega)$ such that
\begin{align*}
        \begin{cases}
          \zeta= 0,\quad \text{in } \Omega\setminus Q_{2R}(x_0),\\
          0\leq \zeta\leq 1,
          \\
           \zeta=1,\quad \text{in }Q_R(x_0),
           \\
           \normmm{\nabla \zeta }\leq \frac2{R}.
            \end{cases}
\end{align*}
 We also define $\mmu:=\frac{\int_{Q_{2R}(x_0)}u\dx}{\normmm{Q_{2R}(x_0)}}$. By testing \eqref{weakform} with $w=\zeta^2(u-\mmu)$ we easily infer 
\begin{align*}
    &\sum_{i=1}^N\int_\Omega \normmm{u\wedge \partial_{x_i}(\zeta(u-\mmu))}^2\dx+\sum_{i=1}^N\int_\Omega (u\wedge \partial_{x_i}u, \zeta^2(u-\mmu)\wedge \partial_{x_i}u)\dx\\&=\sum_{i=1}^N\int_\Omega\normmm{u\wedge (u-\mmu)\partial_{x_i}\zeta}^2\dx+\int_\Omega g\cdot (\zeta^2(u-\mmu))\dx.
\end{align*}
Note now that 
\begin{align*}
\sum_{i=1}^N\int_\Omega \normmm{u\wedge \partial_{x_i}(\zeta(u-\mmu))}^2\dx= \frac 1 N \int_\Omega\normmm{\nabla (\zeta(u-\mmu)))}^2\dx+\sum_{i=1}^N\int_\Omega \normmm{\bP u\wedge \partial_{x_i}(\zeta(u-\mmu))}^2\dx.
\end{align*}
Also, we get
\begin{align*}
    \normmm{\int_\Omega g\cdot (\zeta^2(u-\mmu))\dx}= \normmm{\int_\Omega g\cdot (\zeta^2(\bP u-\bP \mmu))\dx}\leq CM\int_{Q_{2R}(x_0)} \normmm{g}\dx,
\end{align*}
as well as, recalling $\normmm{u}\leq \sqrt{M^2+\frac1{N}}$,
\begin{align*}
    \sum_{i=1}^N\int_\Omega\normmm{u\wedge (u-\mmu)\partial_{x_i}\zeta}^2\dx=\sum_{i=1}^N\int_\Omega\normmm{u\wedge (u-\mmu)\partial_{x_i}\zeta}^2\dx\leq \frac{C(M)}{R^2}\int_{Q_{2R}(x_0)}\normmm{u-\mmu}^2\dx.
\end{align*}
Now the critical term, for which we exploit ${\sum_{i=1}^N\int_\Omega \zeta^2\normmm{(u-\mmu)\wedge \partial_{x_i}u}^2\dx}\geq0$ : we have, recalling $u-\mmu\in T\Sigma$, 
\begin{align*}
    &{\sum_{i=1}^N\int_\Omega (u\wedge \partial_{x_i}u, \zeta^2(u-\mmu)\wedge \partial_{x_i}u)\dx}\\&
    ={\sum_{i=1}^N\int_\Omega \zeta^2\normmm{(u-\mmu)\wedge \partial_{x_i}u}^2\dx}+{\sum_{i=1}^N\int_\Omega (\bP \mmu\wedge \partial_{x_i}u, \zeta^2(u-\mmu)\wedge \partial_{x_i}u)\dx}\\&\geq {\sum_{i=1}^N\int_\Omega \zeta^2\normmm{(u-\mmu)\wedge \partial_{x_i}u}^2\dx}-{\sum_{i=1}^N\int_\Omega \zeta^2\normmm{(u-\mmu)\wedge \partial_{x_i}u}^2\dx}-\frac14{\sum_{i=1}^N\int_\Omega \zeta^2\normmm{\bP \mmu\wedge \partial_{x_i}u}^2\dx}\\&
    =-\frac14{\sum_{i=1}^N\int_\Omega \zeta^2\normmm{\bP \mmu\wedge \partial_{x_i}u}^2\dx},
\end{align*}
and this last term can be estimated as follows, for any $q>0$,
\begin{align*}
    &\frac14{\sum_{i=1}^N\int_\Omega \zeta^2\normmm{\bP \mmu\wedge \partial_{x_i}u}^2\dx}= \frac14{\sum_{i=1}^N\int_\Omega \normmm{\bP \mmu\wedge \partial_{x_i}(\zeta (u-\mmu))-\bP \mmu\wedge (u-\mmu)\partial_{x_i}\zeta}^2\dx}\\&
    \leq \left(\frac14+\frac q{16}\right)\sum_{i=1}^N\int_\Omega \normmm{\bP \mmu\wedge \partial_{x_i}(\zeta (u-\mmu))}^2\dx+\left(\frac14+\frac1 q   \right)\sum_{i=1}^N\int_\Omega \normmm{\bP \mmu\wedge (u-\mmu)\partial_{x_i}\zeta }^2\dx\\&
    \leq \left(\frac14+\frac q{16}\right)\normmm{\bP\mmu}^2\sum_{i=1}^N\int_\Omega\normmm{\partial_{x_i}(\zeta(u-\mmu))}^2\dx+\frac{C(M,\frac1q)}{R^2}\int_{Q_{2R}(x_0) }\normmm{u-\mmu}^2\dx\\&
    \leq \left(\frac14+\frac q{16}\right)M^2\sum_{i=1}^N\int_\Omega\normmm{\partial_{x_i}(\zeta(u-\mmu))}^2\dx+\frac{C(M,\frac1q)}{R^2}\int_{Q_{2R}(x_0) }\normmm{u-\mmu}^2\dx.
\end{align*}
Assuming now $M<\sqrt\frac{ 4}{ N}$, we can rearrange the terms to obtain 
\begin{align*}
  &\left(\tfrac 1 N-M^2\left(\frac14+\frac q{16}\right)     \right)\sum_{i=1}^N \int_\Omega\normmm{\partial_{x_i}(\zeta(u-\mmu))}^2\dx+\sum_{i=1}^N\int_\Omega \normmm{\bP u\wedge \partial_{x_i}(\zeta(u-\mmu))}^2\dx\\&
  \leq   \frac{C(M,\frac1q)}{R^2}\int_{Q_{2R}(x_0) }\normmm{u-\mmu}^2\dx+CM\int_{Q_{2R}(x_0)} \normmm{g}\dx,
\end{align*}
and observe that, choosing $q>0$ such that
\begin{align*}
     \frac1 N-M^2\left(\frac14+\frac q{16}\right)=\epsilon>0,    
\end{align*}
we get 
\begin{align}
  \non&\epsilon\sum_{i=1}^N\int_\Omega \normmm{\partial_{x_i}(\zeta(u-\mmu))}^2\dx+\sum_{i=1}^N\int_\Omega \normmm{\bP u\wedge \partial_{x_i}(\zeta(u-\mmu))}^2\dx\\&
  \leq   \frac{C(M,\frac1q)}{R^2}\int_{Q_{2R}(x_0) }\normmm{u-\mmu}^2\dx+CM\int_{Q_{2R}(x_0)} \normmm{g}\dx,\label{pp}
\end{align}
which entails the result of interior integrability $\bW^{1,p}(\Omega')$, for some $p>2$, for any $\Omega'\subset\subset \Omega$, arguing, for instance, as in \cite[Lemma 4.1]{Garcke}, namely after applying the Sobolev-Poincaré inequality on cubes (cf., e.g., \cite[Theorem A.2]{Garcke}) and then using \cite[Proposition A.1]{Garcke}, whose proof was first proposed in \cite{GiaquintaModica}. To show the higher integrability of $\nabla u$
at the boundary, since $\partial\Omega$ is of Lipschitz class, this can be obtained choosing $x_0\in \partial \Omega$, arguing similarly as to obtain \eqref{pp}, and then concluding following the proof of \cite[Theorem 4.1]{Garcke}. This concludes the proof of \eqref{higherint}.
\end{proof}
To obtain higher-order integrability for $\nabla u_h$, recalling \eqref{inclusion4}, we aim at applying Lemma \ref{higher}, with $g:=-\partial_t \widehat u_h-\Psi'(u_h)\in \bL^2(\Omega)$, using $\norm{u_h}_{L^\infty(0,\infty;\bH^1(\Omega))}\leq C$ as well as \eqref{estimates3} and \eqref{boundu}. This gives
\begin{align}
\norm{u_h}_{\bW^{1,p}(\Omega)}\leq C(M)(1+\norm{\partial_t\widehat u_h}_{\bL^2(\Omega)}),\label{ctrl0}
\end{align}
for some $p>2$, and thus 
\begin{align}
\norm{u_h}_{\bL^2(0,T;\bW^{1,p}(\Omega))}\leq C(M).
\label{ctrl2}
\end{align}
 Therefore, by interpolation with $L^\infty(0,\infty;\bH^1(\Omega))$ we finally infer that there exists $r>2$ such that 
\begin{align}
\norm{u_h}_{L^r(0,T;\bW^{1,r}(\Omega))}\leq C(M),
\label{ctrl1}
\end{align}
for any $T>0$. This gives of course
\begin{align}
u_h\rightharpoonup u,\quad\text{ weakly in }L^r(0,T;\bW^{1,r}(\Omega)),
    \label{precompr}
\end{align}
as $h\to0$. 

Arguing in the same way, we can also obtain higher integrability for $\nabla u^h$, using \eqref{inclusion5b} and Lemma \ref{higher} with $g=-v^h-\Psi'(u^h)$. In particular, also using \eqref{estimates1}, we get, for almost any $t\geq 0$, 
\begin{align}
	\Vert{u^h}\Vert_{\bW^{1,p}(\Omega)}\leq C(M)\Big(1+\Vert{v^h}\Vert_{\bL^2(\Omega)}\Big),\label{ctrl0A}
\end{align}
as well as, using \eqref{estimates3} for $v^h$, for some $s>2$, 
\begin{align}
	\Vert{u^h}\Vert_{L^s(0,T;\bW^{1,s}(\Omega))}\leq C(M),
	\label{ctrl1bb}
\end{align}
for any $T>0$, and thus also 
\begin{align}
	u^h\rightharpoonup u,\quad\text{ weakly in }L^s(0,T;\bW^{1,s}(\Omega)),
	\label{precompr1}
\end{align}
as $h\to0$, where we exploited \eqref{DeGiorgiA}-\eqref{DeGiorgiB} with $z=u$ by \eqref{uhA}.
\subsection{Strong convergence of the gradients}
In this section, we need to show that $\nabla u_h$ strongly converges to $\nabla u$ in $\bL^2(\Omega\times(0,T))$, for any $T>0$, as $h\to0$. To this end, we choose $w=u_h-u$ in \eqref{inclusion5}, which is possible thanks to \eqref{bounds} and \eqref{boundu}. This gives
\begin{align}
&\int_\Omega \partial_t \widehat u_h\cdot (u_h-u)\dx+\sum_{i=1}^N\int_\Omega ((u_h-u)\wedge \partial_{x_i}u_h,u_h\wedge   \partial_{x_i}u_h)\non\dx\\&+\sum_{i=1}^N\int_\Omega (u_h\wedge \partial_{x_i}(u_h-u),u_h\wedge   \partial_{x_i}u_h)\dx+\int_\Omega\Psi'(u_h)\cdot(u_h-u)\dx=0 ,\quad \text{for any }t\geq 0.\label{inclusion6}
\end{align}
Now, we have, since $u_h,u\in \Sigma$ almost everywhere in $\Omega\times(0,\infty)$,
\begin{align*}
  \frac1N\norm{\nabla (u-u_h)}^2_{\bL^2(\Omega)}\leq &\sum_{i=1}^N\int_\Omega (u_h\wedge \partial_{x_i}(u_h-u),u_h\wedge   \partial_{x_i}(u_h-u))\dx,
\end{align*}
so that, using \eqref{inclusion6}, we infer 
\begin{align}
 \non\frac1N\norm{\nabla (u-u_h)}^2_{\bL^2(\Omega)}&\leq\sum_{i=1}^N\int_\Omega (u_h\wedge \partial_{x_i}u,u_h\wedge   \partial_{x_i}(u_h-u))\dx-\int_\Omega \partial_t \widehat u_h\cdot (u_h-u)\dx\\&\quad -\sum_{i=1}^N\int_\Omega ((u_h-u)\wedge \partial_{x_i}u_h,u_h\wedge   \partial_{x_i}u_h)\dx-\int_\Omega\Psi'(u_h)\cdot(u_h-u)\dx.\label{inclusion8}
\end{align}
As a consequence of the previous regularity properties of the solution, we can estimate the right-hand side as follows. First, by Cauchy-Schwarz and Young's inequalities, we infer 
\begin{align*}
  &  \normmm{\int_\Omega \partial_t \widehat u_h\cdot (u_h-u)\dx+\int_\Omega\Psi'(u_h)\cdot(u_h-u)\dx}\\&\leq (\norm{\partial_t\widehat u_h}_{\bL^2(\Omega)}+\norm{\Psi'(u_h)}_{\bL^2(\Omega)})\norm{u-u_h}_{\bL^2(\Omega)}.
\end{align*}
Then by H\"older's and Young's inequalities, we get, recalling \eqref{estimates1}, \eqref{estimates3}, and \eqref{ctrl0},
\begin{align*}
    &\normmm{\sum_{i=1}^N\int_\Omega ((u_h-u)\wedge \partial_{x_i}u_h,u_h\wedge   \partial_{x_i}u_h)\non\dx}\\&\leq C\norm{\nabla u_h}_{\bL^p(\Omega)}\norm{\nabla u_h}_{\bL^2(\Omega)}\norm{u_h}_{\bL^\infty(\Omega)}\norm{u-u_h}_{\bL^{\frac{2p}{p-2}}(\Omega)}\\&
    \leq C(M)(1+\norm{\partial_t\widehat u_h}_{\bL^2(\Omega)})\norm{u-u_h}_{\bL^{\frac{2p}{p-2}}(\Omega)}.
\end{align*}
Moreover, we have, recalling \eqref{ctrl1},
\begin{align*}
&\normmm{\sum_{i=1}^N\int_0^T\int_\Omega (u_h\wedge \partial_{x_i}u,u_h\wedge   \partial_{x_i}(u_h-u))\dx\dt}\\&=\normmm{\sum_{i=1}^N\int_0^T\int_\Omega ((u_h-u)\wedge \partial_{x_i}u,u_h\wedge   \partial_{x_i}(u_h-u))\dx\dt}\\&\quad
+\normmm{\sum_{i=1}^N\int_0^T\int_\Omega (u\wedge \partial_{x_i}u,(u_h-u)\wedge   \partial_{x_i}(u_h-u))\dx\dt}\\&\quad 
+\normmm{\sum_{i=1}^N\int_0^T\int_\Omega (u\wedge \partial_{x_i}u,u\wedge   \partial_{x_i}(u_h-u))\dx\dt}\\&
\leq 
C\norm{u_h-u}_{\bL^{\frac{2r}{r-2}}(\Omega\times(0,T))}\norm{\nabla u_h}_{\bL^2(\Omega\times(0,T))}\norm{u_h}_{\bL^\infty(\Omega\times(0,T))}(\norm{\nabla u_h}_{\bL^r(\Omega\times(0,T))}+\norm{\nabla u}_{\bL^r(\Omega\times(0,T))})\\&
\quad +C\norm{u_h-u}_{\bL^{\frac{2r}{r-2}}(\Omega\times(0,T))}\norm{\nabla u_h}_{\bL^2(\Omega\times(0,T))}\norm{u}_{\bL^\infty(\Omega\times(0,T))}(\norm{\nabla u_h}_{\bL^r(\Omega\times(0,T))}+\norm{\nabla u}_{\bL^r(\Omega\times(0,T))})\\&
\quad +\sum_{i=1}^N\normmm{\int_0^T\int_\Omega \normmm{u}^2\partial_{x_i}u\cdot \partial_{x_i}(u_h-u)-\frac12\partial_{x_i}\normmm{u}^2u\cdot \partial_{x_i}(u_h-u)\dx\dt}\\&
\leq 
C(M)\norm{u_h-u}_{\bL^{\frac{2r}{r-2}}(\Omega\times(0,T))}(\norm{\nabla u_h}_{\bL^r(\Omega\times(0,T))}+\norm{\nabla u}_{\bL^r(\Omega\times(0,T))})\\&
\quad +\sum_{i=1}^N\normmm{\int_0^T\int_\Omega \normmm{u}^2\partial_{x_i}u\cdot \partial_{x_i}(u_h-u)-\frac12\partial_{x_i}\normmm{u}^2u\cdot \partial_{x_i}(u_h-u)\dx\dt}.
\end{align*}
Summing up all the estimates, and integrating in time over $(0,T)$, we finally deduce from \eqref{inclusion8} that 
\begin{align*}
    &\frac1N\int_0^T\norm{\nabla (u-u_h)}^2_{\bL^2(\Omega)}\dt\\&\leq \int_0^T(\norm{\partial_t\widehat u_h}_{\bL^2(\Omega)}+\norm{\Psi'(u_h)}_{\bL^2(\Omega)})\norm{u-u_h}_{\bL^2(\Omega)}\dt\\&
\quad+C(M)\int_0^T(1+\norm{\partial_t\widehat u_h}_{\bL^2(\Omega)})\norm{u-u_h}_{\bL^{\frac{2p}{p-2}}(\Omega)}\dt
    \\&
    \quad +C(M)\norm{u_h-u}_{\bL^{\frac{2r}{r-2}}(\Omega\times(0,T))}(\norm{\nabla u_h}_{\bL^r(\Omega\times(0,T))}+\norm{\nabla u}_{\bL^r(\Omega\times(0,T))})\\&
\quad +\sum_{i=1}^N\normmm{\int_0^T\int_\Omega \normmm{u}^2\partial_{x_i}u\cdot \partial_{x_i}(u_h-u)-\frac12\partial_{x_i}\normmm{u}^2u\cdot \partial_{x_i}(u_h-u)\dx\dt}
\\&\leq 
C(M)(1+\norm{\partial_t\widehat u_h}_{\bL^2(\Omega\times(0,T))})(\norm{u-u_h}_{\bL^2(\Omega\times(0,T))}+\norm{u-u_h}_{\bL^2(0,T;\bL^{\frac{2p}{p-2}}(\Omega))})\\&
\quad +C(M)\norm{u_h-u}_{\bL^{\frac{2r}{r-2}}(\Omega\times(0,T))}\\&
\quad +\sum_{i=1}^N\normmm{\int_0^T\int_\Omega \normmm{u}^2\partial_{x_i}u\cdot \partial_{x_i}(u_h-u)-\frac12\partial_{x_i}\normmm{u}^2u\cdot \partial_{x_i}(u_h-u)\dx\dt}
\\&\leq 
C(M)(\norm{u-u_h}_{\bL^2(\Omega\times(0,T))}+\norm{u-u_h}_{\bL^2(0,T;\bL^{\frac{2p}{p-2}}(\Omega))}+\norm{u_h-u}_{\bL^{\frac{2r}{r-2}}(\Omega\times(0,T))})\\&
\quad +\sum_{i=1}^N\normmm{\int_0^T\int_\Omega \normmm{u}^2\partial_{x_i}u\cdot \partial_{x_i}(u_h-u)-\frac12(\partial_{x_i}\normmm{u}^2)u\cdot \partial_{x_i}(u_h-u)\dx\dt},
\end{align*}
where we used \eqref{estimates3}, \eqref{bounds},    \eqref{boundu}, and \eqref{ctrl1}. We can now conclude that the right-hand side of the inequality above converges to zero as $h\to0$, recalling \eqref{weakH1}, \eqref{uh} and the bound \eqref{boundu}. Note that the last term converges to zero thanks to the weak convergence \eqref{weakH1} and the fact that $\normmm{u}^2\partial_{x_i}u\in \bL^2(\Omega\times(0,T))$ and $ (\partial_{x_i}\normmm{u}^2) u\in \bL^2(\Omega\times(0,T))$ for any $i=1,\ldots,d$. We can thus conclude that 
\begin{align}
u_h\to u\quad\text{ strongly in }L^2(0,T;\bH^1(\Omega)),
    \label{strongnabla}
\end{align}
for any $T>0$. Of course, recalling the bound \eqref{ctrl1}, we also infer by interpolation that there exits $r_1>2$ such that
\begin{align}
u_h\to u\quad\text{ strongly in }L^{r_1}(0,T;\bW^{1,r_1}(\Omega)).
\label{strongnabla2}
\end{align}
Following the same argument, \textit{mutatis mutandis} we can also obtain from the weak formulation \eqref{inclusion5b}, recalling \eqref{estimates1}-\eqref{estimates3} for $u^h$ and $v^h$, together with \eqref{bounds}, \eqref{DeGiorgiA}-\eqref{DeGiorgiB}, \eqref{uhB}, and \eqref{ctrl1bb}, that
\begin{align}
	u^h\to u\quad\text{ strongly in }L^2(0,T;\bH^1(\Omega)),\quad \forall T>0,
	\label{strongnablaA}
\end{align}
and
\begin{align}
	u^h\to u\quad\text{ strongly in }L^{s_1}(0,T;\bW^{1,s_1}(\Omega)),
	\label{strongnabla2A}
\end{align}
for some $s_1>2$, for any $T>0$.

\subsection{Limit as $h\to 0$}
We can finally pass to the limit in \eqref{inclusion5}, after an integration in time over $(0,T)$, recalling \eqref{bounds}, \eqref{boundu}, the convergences \eqref{c1ab}-\eqref{uh}, and \eqref{strongnabla}. In particular, concerning the two critical nonlinear terms, recalling that $w\in \bH^1(\Omega;T\Sigma)\cap \bL^\infty(\Omega)$ we have by H\"older's inequality that
\begin{align*}
    &\normmm{\int_0^T\int_\Omega (w\wedge \partial_{x_i}u_h,u_h\wedge   \partial_{x_i}u_h)\non\dx\dt-\int_0^T\int_\Omega (w\wedge \partial_{x_i}u,u\wedge   \partial_{x_i}u)\non\dx\dt}\\&
    \leq \normmm{\int_0^T\int_\Omega (w\wedge \partial_{x_i}(u_h-u),u_h\wedge   \partial_{x_i}u_h)\non\dx\dt}\\&
    \quad +\normmm{\int_0^T\int_\Omega (w\wedge \partial_{x_i}u,(u_h-u)\wedge   \partial_{x_i}u_h)\non\dx\dt}\\&\quad 
    + \normmm{\int_0^T\int_\Omega (w\wedge \partial_{x_i}u,u\wedge   \partial_{x_i}(u_h-u))\non\dx\dt}\\&
    \leq 
    \norm{w}_{\bL^\infty(\Omega\times(0,T))}\norm{u_h}_{\bL^\infty(\Omega\times(0,T))}\norm{\partial_{x_i}u_h}_{\bL^2(\Omega\times(0,T))}\norm{\partial_{x_i}(u-u_h)}_{\bL^2(\Omega\times(0,T))}\\&\quad +\norm{w}_{\bL^\infty(\Omega\times(0,T))}\norm{u_h-u}_{\bL^{\frac{2r}{r-2}(\Omega\times(0,T))}(\Omega\times(0,T))}\norm{\partial_{x_i}u}_{\bL^2(\Omega\times(0,T))}\norm{\partial_{x_i}u_h}_{\bL^r(\Omega\times(0,T))}\\&
    \quad +\norm{w}_{\bL^\infty(\Omega\times(0,T))}\norm{u}_{\bL^\infty(\Omega\times(0,T))}\norm{\partial_{x_i}u}_{\bL^2(\Omega\times(0,T))}\norm{\partial_{x_i}(u-u_h)}_{\bL^2(\Omega\times(0,T))}\\&
    \leq C(M)\norm{w}_{\bL^\infty(\Omega\times(0,T))}\left(\norm{\partial_{x_i}(u-u_h)}_{\bL^2(\Omega\times(0,T))}+\norm{u-u_h}_{\bL^\frac{2r}{r-2}(\Omega\times(0,T))}\right)\to0,
\end{align*}
as $h\to 0$, where we crucially used the controls \eqref{estimates1}, \eqref{bounds}, \eqref{boundu}, and \eqref{ctrl1}, together with the strong convergences \eqref{uh}, \eqref{strongnabla}.

Analogously, concerning the similar term in \eqref{inclusion5}, since $w\in \bH^1(\Omega;T\Sigma)\cap \bL^\infty(\Omega)$, we have
\begin{align*}
&\normmm{\int_0^T\int_\Omega (u_h\wedge \partial_{x_i}w,u_h\wedge   \partial_{x_i}u_h)\dx\dt-\int_0^T\int_\Omega (u\wedge \partial_{x_i}w,u\wedge   \partial_{x_i}u)\dx\dt}\\&
\leq C(M)\norm{\partial_{x_i}w}_{\bL^2(\Omega\times(0,T))}\left(\norm{\partial_{x_i}(u_h-u)}_{\bL^2(\Omega\times(0,T))}+\norm{u_h-u}_{\bL^\frac{2r}{r-2}(\Omega\times(0,T))}\right)\to0.
\end{align*}
We can thus pass to the limit as $h\to0$ in \eqref{inclusion5}, integrated in time over $(0,T)$, to finally obain, after a time localization, that $u$ satisfies equation \eqref{inclusion2}.

Analogously, exploiting \eqref{DeGiorgiA}-\eqref{DeGiorgiB}, \eqref{c3b}, and \eqref{strongnablaA}-\eqref{strongnabla2A}, we can also pass to the limit as $h\to0$ in \eqref{inclusion5b}, obtaining 
\begin{align}
	&\int_\Omega  v\cdot w\dx+\sum_{i=1}^N\int_\Omega (w\wedge \partial_{x_i}u,u\wedge   \partial_{x_i}u)\non\dx\\&+\sum_{i=1}^N\int_\Omega (u\wedge \partial_{x_i}w,u\wedge   \partial_{x_i}u)\dx+\int_\Omega\Psi'(u)\cdot w\dx=0,\label{inclusion5bc}
\end{align}   
which entails, after comparison with \eqref{inclusion2}, that, for almost any $t\geq0$,
$$
\int_\Omega  v\cdot w\dx=\int_\Omega  \pt u\cdot w\dx,\quad \forall w\in \bH^1(\Omega;T\Sigma)\cap \bL^\infty(\Omega),
$$
and thus we can finally identify $v=\pt u$ almost everywhere in $\Omega\times(0,\infty)$.
\subsection{Sharp energy dissipation inequality}
We are left with showing the validity of the sharp energy dissipation inequality \eqref{energyineq1}, which is possible due to the use of De Giorgi's interpolation. First, note that we can rewrite \eqref{energy ineq} as 
\begin{align}
	\widehat\cE_1(u_h(t))+\int_\Omega \Psi(u_h(t))\dx
    &+\frac12\int_{\ceil{\frac s h}h}^{\floor{\frac th}h}\norm{\pt \widehat u_h}_{\bL^2(\Omega)}^2\d\tau\non
    +\frac12\int_{\ceil{\frac s h}h}^{\floor{\frac th}h}\norm{v_h}^2_{\bL^2(\Omega)}\d\tau\non
    \\&\leq \widehat\cE_1(u_h(s))+\int_\Omega \Psi(u_h(s))\dx.
	\label{energy ineqA}
\end{align}
Moreover, choosing $k=0$ in \eqref{energy ineq}, we also get 
\begin{align}
	\widehat\cE_1(u_h(t))+\int_\Omega \Psi(u_h(t))\dx
    &+\frac12\int_{0}^{\floor{\frac th}h}\norm{\pt \widehat u_h}_{\bL^2(\Omega)}^2\d\tau\non
    +\frac12\int_{0}^{\floor{\frac th}h}\norm{v_h}^2_{\bL^2(\Omega)}\d\tau\non
    \\&\leq \widehat\cE_1(u_0)+\int_\Omega \Psi(u_0)\dx.
	\label{energy ineqB}
\end{align}
 Now, let us notice that from the uniform bounds \eqref{bounds0} together with the strong convergences \eqref{strongnabla}-\eqref{strongnabla2}, which give
$$
u_h(t)\to u,\quad\text{ strongly in }\bH^1(\Omega),\quad\text{ for a.a. }t\geq0,
$$
we can easily obtain that 
\begin{align*} \widehat{\mE}_{1}(u_h(s))\to \widehat{\mathcal E}_{1}(u(s)),\text{ as }h\to0,
\end{align*}   
for almost any $s\geq0$. Therefore, the right-hand side of \eqref{energy ineqA} converges to  $\widehat \cE_{1}  {(}u(s)  {)}+\int_\Omega \Psi(u(s))\dx$ for almost any $s\geq0$. Similarly, the energy term in the left-hand side of \eqref{energy ineqA} converges as $h\to0$ for almost any $t\geq0$. 

Then, by lower semicontinuity of the norms involved, recalling the weak convergences \eqref{c3a}-\eqref{c3b}, with the identifications $w=v=\pt u$ shown in the previous sections, we infer
\begin{align*}
	&	\int_s^t\norm{\pt u}^2_{\bL^2(\Omega)}\d\tau\leq \liminf_{h\to0} 	\int_{\ceil{\frac sh}h}^{\floor{\frac th}h}\norm{\pt \widehat u_h}^2_{\bL^2(\Omega)}\d\tau,\\&
	\int_s^t\norm{\pt u}^2_{\bL^2(\Omega)}\d\tau\leq \liminf_{h\to0} 	\int_{\ceil{\frac sh}h}^{\floor{\frac th}h}\norm{v_h}^2_{\bL^2(\Omega)}\d\tau,
\end{align*}
for any $s,t\geq0$. 

As a consequence, we can pass to the limit in \eqref{energy ineqA} and \eqref{energy ineqB}, finally obtaining \eqref{energyineq1} for almost any $0\leq s<t$, with $s=0$ included. In order to further obtain the inequality \textit{for any }$t>0$, it is enough to notice that, by the regularity of $u$ and standard embedding theorems, it holds $u\in C([0,\infty);\bL^p(\Omega))$ for any $p\geq 2$, as well as $u\in C_w([0,\infty);\bH^1(\Omega))$, i.e., weakly continuous with values in $\bH^1(\Omega)$. As a consequence, we can argue   {as in the argument for} \eqref{lsc} to infer that the map 
\begin{align*}
t\mapsto \widehat\cE_1(u(t)) \text{ is lower semicontinuous on $[0,\infty)$},
\end{align*}
which allows to immediately deduce, by the lower semicontinuity in $t$ of the other integrals in the left-hand side of \eqref{energyineq1}, that the energy inequality holds for any $t>0$ and for almost any $0\leq s<t$, with $s=0$ included. This concludes the proof of Theorem \ref{thm:existenceweak}.
\\

\textbf{Acknowledgments.} Part of this contribution was completed while AP was visiting HG at the Faculty of Mathematics of the University of Regensburg, and TL at the University of Heidelberg, whose hospitality is kindly acknowledged. AP also gratefully aknowledges support from the Alexander von Humboldt Foundation for the stay in Regensburg.  
AP is a member of Gruppo Nazionale per l’Analisi Matematica, la Probabilità e le loro Applicazioni (GNAMPA) of
Istituto Nazionale per l’Alta Matematica (INdAM).  

\bibliographystyle{siam}
	\bibliography{multiAC}
	
\end{document}